\pdfoutput=1
\RequirePackage{ifpdf}
\ifpdf 
\documentclass[pdftex]{sigma}
\else
\documentclass{sigma}
\fi

\numberwithin{equation}{section}

\newtheorem{Theorem}{Theorem}[section]
\newtheorem{Corollary}[Theorem]{Corollary}
\newtheorem{Lemma}[Theorem]{Lemma}
\newtheorem{Proposition}[Theorem]{Proposition}
 { \theoremstyle{definition}

\newtheorem{Example}[Theorem]{Example}
\newtheorem{Remark}[Theorem]{Remark} }

\newcommand{\mgg}{\mathfrak g }

\newcommand{\C}{\mathbb C}

\newcommand{\F}{\mathbb F}

\newcommand{\func}[1]{\operatorname{#1}}

\begin{document}

\allowdisplaybreaks

\newcommand{\arXivNumber}{1811.07854}

\renewcommand{\PaperNumber}{051}

\FirstPageHeading

\ShortArticleName{De Rham 2-Cohomology of Real Flag Manifolds}

\ArticleName{De Rham 2-Cohomology of Real Flag Manifolds}

\Author{Viviana DEL BARCO~$^{\dag\ddag}$ and Luiz Antonio Barrera SAN MARTIN~$^\ddag$}

\AuthorNameForHeading{V.~del Barco and L.A.B.~San Martin}

\Address{$^\dag$~UNR-CONICET, Rosario, Argentina}
\EmailD{\href{mailto:delbarc@fceia.unr.edu.ar}{delbarc@fceia.unr.edu.ar}}
\URLaddressD{\url{http://www.fceia.unr.edu.ar/~delbarc/}}

\Address{$^\ddag$~IMECC-UNICAMP, Campinas, Brazil}
\EmailD{\href{mailto:smartin@ime.unicamp.br}{smartin@ime.unicamp.br}}

\ArticleDates{Received January 08, 2019, in final form June 25, 2019; Published online July 05, 2019}

\Abstract{Let $\mathbb{F}_{\Theta }=G/P_{\Theta }$ be a flag manifold associated to a non-compact real simple Lie group $G$ and the parabolic subgroup $P_{\Theta }$. This is a closed subgroup of $G$ determined by a subset $\Theta $ of simple restricted roots of $\mathfrak{g}=\operatorname{Lie}(G)$. This paper computes the second de Rham cohomology group of $\mathbb{F}_\Theta$. We prove that it is zero in general, with some rare exceptions. When it is non-zero, we give a basis of $H^2(\mathbb{F}_\Theta,\mathbb{R})$ through the Weil construction of closed 2-forms as characteristic forms of principal fiber bundles. The starting point is the computation of the second homology group of $\mathbb{F}_{\Theta }$ with coefficients in a ring $R$.}


\Classification{57T15; 14M15}

\section{Introduction}

A real flag manifold is a homogeneous manifold $\mathbb{F}_\Theta=G/P_\Theta$ where $G$ is a connected Lie group with Lie algebra $\mathfrak{g}$ which is non-compact and semi-simple, and $P_\Theta$ is a parabolic subgroup of~$G$. Real grassmannians and projective spaces belong to the family of real flag manifolds.

Topological properties of these manifolds have been of interest for several authors. The fundamental group of real flag manifolds was considered by Wiggerman~\cite{Wi} giving a continuation to the work of Ehresmann on real flag manifolds with $G=\mathrm{SL}(n,\mathbb{R})$. The integral homology of real flag manifolds has been studied by Kocherlakota who gives an algorithm for its computation through Morse homology~\cite{Ko}, based on previous work of Bott and Samelson~\cite{BS}. A different approach to study their homology is given by Rabelo and the second named author of this paper~\cite{RSM}, focusing on the geometry involved in the cellular decomposition of the manifolds. Mare \cite{Ma} considered the cohomology rings of a subfamily of real flag manifolds, namely, those having all roots of rank greater of equal than two. Real flag manifolds of split real forms are not part of the subfamily considered in~\cite{Ma}.

The present paper focuses on the computation of the second de Rham cohomology group of real flag manifolds. Our motivation comes from symplectic geometry; it is well known that the annihilation of the second de Rham cohomology group is an obstruction to the existence of symplectic structures on compact manifolds. It is worth stressing that we deal with real flag manifolds associated to any non-compact real form $G$ of complex simple Lie groups, and any parabolic subgroup, without restrictions.

To obtain the cohomology groups, we start with the explicit computation of second homology groups with coefficients on a ring $R$, $H_{2}(\mathbb{F}_\Theta,R)$. The description of $H_{2}(\mathbb{F}_\Theta,R)$ does not follow directly from the works of Kocherlakota and Rabelo--San Martin. Instead it requires to work over the root systems. The classification of the homology groups is achieved in Theorem \ref{teoResumo}, after developments in Sections~\ref{section2} and~\ref{sec.comput} where we apply the tools of Rabelo--San Martin~\cite{RSM}. Mainly, we show that $H_{2}(\mathbb{F}_\Theta,R)$ is a torsion group, except when there are roots of rank $2$ in the system of restricted roots of the real flag. The rank of a root $\alpha $ is $\operatorname{rank} \alpha =\dim \mathfrak{g}_{\alpha }+\dim \mathfrak{g}_{2\alpha }$, where these subspaces are the root spaces associated to $\alpha$ and $2\alpha$ (if it is a root).

It follows that $H_{\rm dR}^{2}(\mathbb{F}_{\Theta },\mathbb{R})=\{0\}$ unless the root system of the real Lie algebra contains roots having rank $2$. Moreover, the number of such roots gives the dimension of $H_{\rm dR}^{2}(\mathbb{F}_{\Theta },\mathbb{R})$ (see Theorem~\ref{tedeRham} below). These data can be read off from the classification table of real simple Lie algebras (see Warner \cite[pp.~30--32]{WarG}) and hence $H_{\rm dR}^{2}(\mathbb{F}_{\Theta },\mathbb{R})$ can be completely determined. We summarize the computation of the second de Rham cohomologies in Theorem~\ref{teoResumo}.

Once we have the classification of the real flag manifolds $\mathbb{F}_{\Theta }$ satisfying $H_{\rm dR}^{2}(\mathbb{F}_{\Theta },\mathbb{R})\neq \{0\}$ we search for differential $2$-forms representing a basis of these non-trivial spaces.

We get such a basis by applying the Weil construction to the principal fiber bundle $\pi \colon$ $K\longrightarrow \mathbb{F}_{\Theta }$ with structure group $K_{\Theta }=P_{\Theta }\cap K$, where $K$ is a maximal compact subgroup of~$G$, and $\mathbb{F}_{\Theta }=G/P_{\Theta }=K/K_{\Theta }$. To perform the Weil construction we choose in a standard way a left invariant connection $\omega $ in the principal bundle $K\longrightarrow \mathbb{F}_{\Theta }$ with curvature form $\Omega $. Each adjoint invariant $f$ in the dual $\mathfrak{k}_{\Theta }^{\ast }$ of the Lie algebra $\mathfrak{k}_{\Theta }$ of~$K_{\Theta }$ yields an invariant closed $2$-form $\widetilde{f}$ in~$K/K_{\Theta }$. We prove that the $2$-forms $\widetilde{f}$ exhaust the $2$-cohomology classes by exhibiting a basis formed by characteristic forms which is dual to the Schubert cells that generate the second real homology.

To prove the surjectivity of the map $f\rightarrow \widetilde{f}$ it is required a careful description of the center of~$K_{\Theta }$ which we provide in Section~\ref{seccenterM} by looking first at the center of the $M$ group where $M=K_{\varnothing }$ is the isotropy group of the maximal flag manifold. In Section~\ref{secCharacClass} we apply the previous results to get the desired dual bases of differential $2$-forms in Theorem~\ref{teo.baseZ}. In Section~\ref{secEsseu} we illustrate our results with concrete computations in the flag manifolds of the real simple Lie algebras $\mathfrak{su}(p,q) $, $p\leq q$, that are real forms of $\mathfrak{sl} ( p+q,\mathbb{C} ) $ and realize the Lie algebras of types $\mathrm{AIII}_{1}$ and $\mathrm{AIII}_{2}$.

As consequence of our results, we obtain that a real flag manifold does not carry symplectic structures, unless its corresponding root system contains roots of rank 2. Moreover, if the system contains roots of rank 2, the only case where the manifold is symplectic is when the real flag manifold is actually the product of complex flag manifolds of the form ${\rm SU}(n)/T$.

\section{Cellular decomposition and boundary maps}\label{section2}

This section aims to fix notations and to introduce the preliminaries for the rest of the paper. For the classical theory the reader is referred to the books of Knapp~\cite{knapp}, Helgason~\cite{He} and Warner~\cite{WarG}. The treatment of the cellular decomposition and the boundary maps of real flag manifolds follows the presentation given in the work of L.~Rabelo and L.~San Martin~\cite{RSM}.

Let $\mathfrak{g }$ be a non-compact real simple Lie algebra and let $\mathfrak{g }=\mathfrak{k }\oplus\mathfrak{s }$ be a Cartan decomposition. Let $\mathfrak{a }$ be a maximal abelian subalgebra of $\mathfrak{s }$ and denote $\Pi$ the set of restricted roots of the pair~$(\mathfrak{g },\mathfrak{a })$. Let $\Sigma$ be a subset of simple roots and denote $\Pi^{\pm}$ the set of positive and negative roots, respectively. The Iwasawa decomposition of $\mathfrak{g }$ is given by $\mathfrak{g }=\mathfrak{k }\oplus\mathfrak{a }\oplus\mathfrak{n }$ with $\mathfrak{n }=\sum\limits_{\alpha\in\Pi^+}\mathfrak{g }_\alpha$ and~$\mathfrak{g}_\alpha$ the
root space corresponding to $\alpha$.

Given a simple Lie group $G$ with Lie algebra $\mathfrak{g}$, denote $K$, $A$ and $N$ the connected subgroups corresponding to the Lie subalgebras $\mathfrak{k}$, $\mathfrak{a}$ and $\mathfrak{n}$, respectively.

To a subset of simple roots $\Theta \subset \Pi $ there is associated the parabolic subalgebra
\begin{gather*}
\mathfrak{p}_{\Theta }=\mathfrak{a}\oplus \mathfrak{m}\oplus \sum_{\alpha \in \Pi ^{+}}\mathfrak{g}_{\alpha }\oplus \sum_{\alpha \in \langle\Theta\rangle ^{-}}\mathfrak{g}_{\alpha },
\end{gather*}
where $\mathfrak{m}$ is the centralizer of $\mathfrak{a}$ in $\mathfrak{k}$ and $\langle\Theta\rangle ^{-}$ is the set of negative roots generated by $\Theta $. The minimal parabolic subalgebra $\mathfrak{p}$ is obtained for $\Theta =\varnothing $. The normalizer $P_{\Theta }$ of $\mathfrak{p}_{\Theta }$ in $G$ is the standard parabolic subgroup associated to $\Theta $. In this manner, each subset $\Theta \subset \Sigma $ defines the homogeneous manifold $\mathbb{F}_{\Theta }=G/P_{\Theta }$; these homogeneous manifolds are the object of study of this paper. Denote by $b_{\Theta }$ the class of the identity in $F_{\Theta }$, that is, $b_{\Theta }=eP_{\Theta }$. For $\Theta =\varnothing $ the index $\Theta $ is dropped to simplify notations and $\mathbb{F}$ is referred as the maximal flag manifold, also called full or complete flag manifold. When $\Theta \neq \varnothing $, $\mathbb{F}_{\Theta }$ is called partial flag manifold.

Let $\mathcal{W}$ be the Weyl group associated to $\mathfrak{a}$. This is a finite group and is generated by reflec\-tions~$r_{\alpha }$ over the hyperplanes $\alpha (H) =0$ in $\mathfrak{a}$, with $\alpha $ a~simple root and $H$ a regular element. The length $\ell (w)$ of an element $w\in \mathcal{W}$ is the number of simple reflections in any reduced expression as a product $w=r_{\alpha _{1}}\cdots r_{\alpha _{t}}$ of reflections with $\alpha _{i}\in \Sigma $. Denote $\Pi _{w}=\Pi ^{+}\cap w\Pi ^{-}$, the set of positive roots taken to negative by $w^{-1}$. For $\Theta \subset \Sigma $, $\mathcal{W}_{\Theta }$ denotes the subgroup of $\mathcal{W}$ generated by $r_{\alpha} $ with $\alpha \in \Theta $.

The Weyl group is isomorphic the quotient $M^*/M$ of the normalizer $M^*$ of $\mathfrak{a }$ in $K$ over the respective centralizer $M$. It acts, through $M^*$, on $\mathbb{F}_\Theta$ and the left $N$ classes of the orbit $M^* b_\Theta$ give a decomposition of $\mathbb{F}_\Theta$, known as the Bruhat decomposition:
\begin{gather*}
\mathbb{F}_\Theta=\coprod_{w\mathcal{W}_\Theta \in \mathcal{W}/\mathcal{W}_\Theta}N\cdot w b_\Theta.
\end{gather*}
One should notice that $N\cdot w b_\Theta$ does not depend on the choice of the representative, namely, $N\cdot w_1 b_\Theta=N\cdot w_2 b_\Theta$ whenever $w_1\mathcal{W}_\Theta =w_2\mathcal{W}_\Theta$.

A cellular decomposition of the flag manifold $\mathbb{F}_\Theta$ is given by Schubert cells ${\mathcal{S}}_w^\Theta=\operatorname{cl}(N\cdot wb_\Theta)$ with $w$ a representative of $w\mathcal{W}_\Theta$; here $\operatorname{cl}$ denotes closure. The dimension of the Schubert cell defined by $w\mathcal{W}_\Theta$ is
\begin{gather*}
\dim {\mathcal{S}}_w^\Theta =\sum_{\alpha\in \Pi_w\backslash \langle \Theta \rangle^+}\dim \mathfrak{g }_\alpha.
\end{gather*}
For the maximal flag this formula reads $\dim S_w=\sum\limits_{\alpha\in \Pi_w}\dim \mathfrak{g }_\alpha$. If $w\in \mathcal{W}$ has reduced expression $w=r_{\alpha _{1}}\cdots r_{\alpha _{t}}$, then
\begin{gather}
\dim S_{w}=\sum_{i=1}^{t}\dim (\mathfrak{g}_{\alpha _{i}}+\mathfrak{g}_{2\alpha _{i}}) \label{eq.dimSw}
\end{gather}
(see \cite[Corollary 2.6]{Wi}). In particular, $\dim S_w=\ell(w)$ for all $w\in \mathcal{W}$ if $\mgg$ is a split real form. For a simple root $\alpha $, we denote
\begin{gather*}
\operatorname{rank}\alpha :=\dim \mathfrak{g}_{\alpha }+\dim \mathfrak{g}_{2\alpha }.
\end{gather*}
Notice that $\operatorname{rank} \alpha=\dim S_{r_{\alpha }}$. Moreover, $\mathfrak{g}_{2\alpha }=\{0\}$ if $\dim \mathfrak{g}_{\alpha }=1$ so that $\operatorname{rank}\alpha =2$ if and only if $\dim \mathfrak{g}_{\alpha }=2$. If $w=r_{\alpha _{1}}\cdots r_{\alpha _{t}}$ is a reduced expression then $\dim S_{w}\geq \ell (w)=t$ if and only if $\operatorname{rank}\alpha_{i}=1$ for all $i=1,\dots,t$.

At this point we fix, once and for all, reduced expressions of the elements $w$ in $\mathcal{W}$, $w=r_{\alpha_1}\cdots r_{\alpha_t}$, as a product of simple reflections. Such decompositions determine the characteristic maps which describe how Schubert cells are glued to give the cellular decomposition of~$\mathbb{F}_\Theta$~\cite{RSM}.

In what follows we recall the definition of the cellular complex and the boundary map giving the homology of~$\mathbb{F}_\Theta$ with coefficients in a ring~$R$. We give explicit formulas for the boundary maps up to dimension $3$ which will be used in the next section.

Let $\mathcal{C}_{i}$ be the $R$-module freely generated by $\mathcal{S}_{w}$, $w\in \mathcal{W}$ and $\dim S_{w}=i$, for $i=0,\ldots,\dim \mathbb{F}$. Notice that $\mathcal{C}_{0}=R$ since there is just one zero-dimensional cell, namely the origin $\{b\}$. Define
\begin{gather*}
\Sigma _{\mathrm{split}} =\{\alpha \in \Sigma \colon \operatorname{rank}\alpha =\dim \mathfrak{g}_{\alpha }=1\}, \\
\Sigma _{2} =\{\alpha \in \Sigma \colon \operatorname{rank}\alpha =2\}.
\end{gather*}

\begin{Proposition}\label{pro.chains}\quad
\begin{itemize}\itemsep=0pt
\item $\mathcal{C}_1$ is the free module spanned by $S_{r_\alpha}$ with $\alpha\in \Sigma_{\mathrm{split}}$.
\item $\mathcal{C}_{2}$ is the free module spanned by $S_{r_{\alpha}r_{\beta }}$ with $\alpha \neq \beta \in \Sigma _{\mathrm{split}}$ and by $S_{r_{\alpha }}$ with $\alpha \in \Sigma _{2}$.
\item $\mathcal{C}_{3}$ is the free module spanned by $S_{r_{\alpha}r_{\beta }r_{\gamma }}$ with $\beta \neq \alpha \neq \gamma \in \Sigma _{\mathrm{split}}$, by $S_{r_{\alpha }r_{\beta }}$ with $\alpha \in \Sigma_{2} $ and $\beta \in \Sigma _{\mathrm{split}}$, or vice-versa, and by $S_{r_{\alpha }}$, with $\alpha \in \Sigma $ and $\operatorname{rank}\alpha =3$.
\end{itemize}
\end{Proposition}

\begin{proof}By definition $S_{r_{\alpha }}$, $S_{r_{\alpha }r_{\beta }}$, $S_{r_{\alpha}r_{\beta }r_{\gamma }}$ are of dimensions $1$, $2$ and $3$ respectively, when $\alpha,\beta,\gamma \in \Sigma _{\mathrm{split}}$ and $\gamma \neq \alpha \neq \beta $ and $\dim S_{r_{\alpha }}=2$ if $\operatorname{rank}\alpha=\dim \mathfrak{g}_{\alpha }=2$. Similarly one obtains $\dim S_{r_{\alpha}}=3$ if and only if $2\alpha $ is not a~root and $\dim \mathfrak{g}_{\alpha}=3$ or $\dim \mathfrak{g}_{\alpha }=2$ and $\dim \mathfrak{g}_{2\alpha }=1$ since $2\alpha $ is not a~root if $\dim \mathfrak{g}_{\alpha }=1$.

Given $w=r_{\alpha }r_{\beta }$, $\dim S_{w}=3$ only if $\dim (\mathfrak{g}_{\alpha }+\mathfrak{g}_{2\alpha })=2$ and $\dim (\mathfrak{g}_{\beta }+\mathfrak{g}_{2\beta })=1$, or vice-versa. But, $\dim (\mathfrak{g}_{\alpha}+\mathfrak{g}_{2\alpha })=2$ implies $\dim \mathfrak{g}_{\alpha }=2$ and $2\alpha $ is not a root. Thus $\dim S_{w}=3$ implies $\alpha \in \Sigma _{2}$ and $\beta \in \Sigma _{\mathrm{split}}$ or the inverse situation for $%
\alpha $ and $\beta $.
\end{proof}

The chain complex of the partial flag manifold $\F_\Theta$ is constituted by {\em minimal} Schubert cells. For any $w\in \mathcal{W}$ there exists a unique $w_{1}\in w\mathcal{W}_{\Theta }$ such that $\Pi ^{+}\cap w_{1}^{-1}\Pi ^{-}\cap \langle\Theta\rangle =\varnothing$ \cite[Lemma~3.1]{RSM}. Such element is called {\em minimal} in $w\mathcal{W}_{\Theta }$ and satisfies $\dim S_{w}^{\Theta}=\dim S_{w_{1}}$. Deno\-te~$\mathcal{W}_{\Theta }^{\min }$ the set of minimal elements in $\mathcal{W}$ with respect to $\Theta $. For $i=0,\ldots,\dim \mathbb{F}_{\Theta }$ let $\mathcal{C}_{i}^{\Theta }\subset \mathcal{C}_{i}$ be the free $R$ module spanned by $S_{w}$ with $w\in \mathcal{W}_{\Theta}^{\min }$ and $\dim S_{w}=i$. Next we describe the minimal elements giving cells of dimension $\leq 3$.

\begin{Lemma}\label{lm.minimalroots}\quad
\begin{itemize}\itemsep=0pt
\item $w=r_\alpha$ is minimal if and only if $\alpha\notin \Theta$,
\item $w=r_{\alpha }r_{\beta }$ is minimal if and only if $\beta \notin \Theta $ and $r_{\beta }\alpha \notin \langle\Theta\rangle $,
\item $w=r_{\alpha }r_{\beta }r_{\gamma }$ is minimal if and only if $\gamma \notin \Theta $ and $r_{\gamma }\beta,r_{\gamma }r_{\beta }\alpha \notin \langle\Theta\rangle $.
\end{itemize}
\end{Lemma}

\begin{proof}The only positive roots taken into a negative root by $w=r_{\alpha }$ with $\alpha \in \Sigma $ are $\alpha $ and $2\alpha $, when this is a root. So $\Pi ^{+}\cap r_{\alpha }\Pi ^{-}\cap \langle\Theta\rangle^{+}=\varnothing $ if and only if $\alpha \notin \Theta $. If $w=r_{\alpha }r_{\beta }$, with $\alpha,\beta \in \Sigma $, $\alpha \neq
\beta $ then the positive roots taken to negative by $r_{\alpha }r_{\beta }$ are $\beta $, $r_{\beta }\alpha $ and the multiples $2\beta $ and $r_{\beta}(2\alpha )$ when they are actually roots. Hence $w$ is minimal if and only if $\beta \notin \Theta $ and $r_{\beta }\alpha \notin \langle\Theta\rangle $. Similarly, for $w=r_{\alpha}r_{\beta }r_{\gamma }$ with $\alpha \neq \beta \neq \gamma $ we have
\begin{gather*}
w\Pi ^{+}\cap \Pi ^{-}=\{\gamma,2\gamma,r_{\gamma }\beta,2r_{\gamma}\beta,r_{\gamma }r_{\beta }\alpha,2r_{\gamma }r_{\beta }\alpha \}\cap \Pi ^{+}.
\end{gather*}
Thus $w=r_{\alpha }r_{\beta }r_{\gamma }$ is minimal if $\gamma$, $r_{\gamma}\beta$, $r_{\gamma }r_{\beta }\alpha $ are not in $\langle\Theta\rangle $.
\end{proof}

The boundary operator $\partial $ applied to a Schubert cell $\mathcal{S}_{w} $ of dimension~$i$, gives a linear combination of Schubert cells of dimension $i-1$
\begin{gather}
\partial \mathcal{S}_{w}=\sum_{w^{\prime }}c(w,w^{\prime })\mathcal{S}_{w^{\prime }}. \label{eq.boundcoeff}
\end{gather}
The coefficients $c(w,w^{\prime })$ are given in~\cite[Section~2]{RSM}. These coefficients are always $0$ or $\pm 2$ and they behave as follows (see also~\cite{SR}). For $\alpha\in \Sigma$ we denote $\alpha ^{\vee }:=\frac{2\alpha }{ \langle \alpha,\alpha \rangle }$, so that $ \langle \alpha ^{\vee },\beta \rangle $ is an integer for $\beta \in \Sigma$ (see for instance \cite{He}).
\begin{enumerate}\itemsep=0pt
\item \label{onlysupr}Given $w=r_{\alpha _{1}}\cdots r_{\alpha _{t}}\in \mathcal{W}$, one has $c(w,w^{\prime })=0$ for any $w^{\prime }\in \mathcal{W}$ which is not obtained from $w$ by removing a reflection $r_{\alpha _{i}}$.
\item \label{suprult}If $w=r_{\alpha _{1}}\cdots r_{\alpha _{t-1}}r_{\alpha_{t}}$ and $w^{\prime }=r_{\alpha _{1}}\cdots r_{\alpha _{t-1}}$ then $c(w,w^{\prime })=0$.
\item Assume $w=r_{\alpha _{1}}\cdots r_{\alpha _{t}}$ and $w^{\prime}=r_{\alpha _{1}}\cdots \widehat{r}_{\alpha _{i}}\cdots r_{\alpha _{t}}$ are reduced decompositions of $w$ and $w^{\prime }$ respectively. If $\operatorname{rank} \alpha _{i}=1$ then $c(w,w^{\prime })=\pm \big(1-(-1) ^{\sigma ( w,w^{\prime }) }\big) $ where
\begin{gather}
\sigma (w,w^{\prime}) =\sum_{\beta \in \Pi _{u}}\langle \alpha _{i}^\vee,\beta \rangle \dim \mathfrak{g}_{\beta },\qquad u=r_{\alpha _{i+1}}\cdots r_{\alpha_{t}}. \label{forsigmasoma}
\end{gather}%
Otherwise, $c(w,w^{\prime })=0$.
\end{enumerate}

Notice that if $w=r_{\alpha_1}\cdots r_{\alpha_t}$ is a reduced expression, the coefficients $c(w,w^{\prime })$ are non-zero only when, by erasing a reflection $r_{\alpha_i}$ with $i\neq t$, a reduced expression of a root $w^{\prime }$ with $\dim S_{w^{\prime }}=\dim S_w-1$ is obtained.

The choice of the sign $\pm $ in $c(w,w^{\prime })$ is given by the degree of a certain map depending on $w$ and $w^{\prime }$ \cite[Theorem~2.6]{RSM} and it can be determined precisely. The computations here will not need the exact number, so the $\pm $ will be present along the paper.

The next proposition gives the boundary of cells in $\mathcal C_i$ for $i\leq 3$.
\begin{Proposition}\label{pro.boundarymax}Given simple roots $\alpha$, $\beta$, $\gamma $ there are the following expressions for the boundary operator $\partial $:
\begin{enumerate}\itemsep=0pt
\item[$1.$] $\partial S_{r_{\alpha }}=0$.
\item[$2.$] If $\alpha \neq \beta $ then $\partial S_{r_{\alpha }r_{\beta }}=c( r_{\alpha }r_{\beta },r_{\beta}) S_{r_{\beta }}$. In particular, if $S_{r_{\alpha }r_{\beta }}$ is a $3$-cell then $\partial S_{r_{\alpha}r_{\beta }}=0$. $($This is the case when $\alpha \in \Sigma _{2}$ and $\beta\in \Sigma _{\mathrm{split}}$ or vice-versa.$)$
\item[$3.$] If $\alpha \neq \beta $ then
\begin{gather}
\partial S_{r_{\alpha }r_{\beta }r_{\alpha }}=c ( r_{\alpha }r_{\beta}r_{\alpha },r_{\beta }r_{\alpha } ) S_{r_{\alpha }r_{\beta }}.\label{eq.deltaSa}
\end{gather}
\item[$4.$] If $\gamma \neq \alpha $ and $\alpha,\beta,\gamma \in \Sigma _{\mathrm{split}}$ then
\begin{gather}
\partial S_{r_{\alpha }r_{\beta }r_{\gamma }}=c ( r_{\alpha }r_{\beta}r_{\gamma },r_{\alpha }r_{\gamma } ) S_{r_{\alpha }r_{\gamma}}+c ( r_{\alpha }r_{\beta }r_{\gamma },r_{\beta }r_{\gamma } )S_{r_{\beta }r_{\gamma }}. \label{eq.deltaSaba}
\end{gather}
\end{enumerate}
\end{Proposition}

\begin{proof}The boundary of $S_{r_{\alpha }}$ contains only the $0$-cell hence $\partial S_{r_{\alpha }}=0$ regardless the dimension of $S_{r_{\alpha }}$. For the cell $S_{r_{\alpha }r_{\beta }}$ its boundary $\partial S_{r_{\alpha}r_{\beta }}$ has no component in the direction of $S_{r_{\alpha }}$ which obtained by removing the last reflection $r_{\beta }$ from $r_{\alpha}r_{\beta }$. If $\operatorname{rank}\alpha =2$ then $c( r_{\alpha}r_{\beta },r_{\beta }) =0$ because $S_{r_{\beta }}$ has codimension $2 $ in $S_{r_{\alpha }r_{\beta }}$. On the other hand if $\operatorname{rank} \beta =2$ then $c(r_{\alpha }r_{\beta },r_{\beta })=\pm \big( 1-(-1) ^{\sigma ( r_{\alpha }r_{\beta },r_{\beta } ) }\big)=0 $ since $\sigma ( r_{\alpha }r_{\beta },r_{\beta }) =\langle \alpha^\vee,\beta \rangle\dim \mathfrak{g}_{\beta }$ is even. This proves the second statement. The last two statements follow from the fact that $c( r_{\alpha }r_{\beta }r_{\gamma },r_{\alpha }r_{\beta })=0$ (removal of the last reflection) and $c( r_{\alpha }r_{\beta}r_{\gamma },r_{\alpha }r_{\alpha }) =0$ since $r_{\alpha }^{2}=1$ is not a reduced expression.
\end{proof}

The boundary map $\partial_\Theta^{\min}\colon \mathcal{C}_i^\Theta\longrightarrow \mathcal{C}_{i-1}^\Theta$ is defined through the boundary map $\partial$ for the homology of the maximal flag. Given $S_w$, with $w\in \mathcal{W}_\Theta^{\min}$, let $I_w$ be the set of minimal elements~$w^{\prime }$ such that $\dim S_{w^{\prime }}=\dim S_w-1$. Define
\begin{gather*} \partial^{\min}_\Theta S_w=\sum_{w^{\prime }\in I_w}c(w,w^{\prime })S_{w^{\prime }}.\end{gather*} The homology
of $\mathbb{F}_\Theta$ with coefficients in $R$ is the homo\-logy of the complex $\big(\mathcal{C}^\Theta,\partial^{\min}_\Theta\big)$ \cite[Theorem~3.4]{RSM}. Proposition~\ref{pro.boundarymax} gives the boundary map of minimal cells up to dimension three.

Propositions \ref{pro.chains} and~\ref{pro.boundarymax} and Lemma~\ref{lm.minimalroots} account to a description of the second homology group of a real flag manifold in terms of its split part and the roots of rank $2$ in the system. In the text below we include the maximal flag case $\mathbb F$ as the particular instance $\mathbb F_\Theta$ with $\Theta=\varnothing$.

Denote $\mathfrak{g }_{\mathrm{split}}$ the split real form whose Dynkin diagram is given by $\Sigma_{\mathrm{split}}$. Set $\Theta_{\mathrm{split}}=\Theta\cap\Sigma_{\mathrm{split}}$ and $\Sigma _{2}^{\Theta }=\{\alpha \in \Sigma \backslash \Theta \colon \operatorname{rank}\alpha =2\}$ and let $\mathbb{F}_{\Theta_{\mathrm{split}}}$ be the flag manifold of $\mathfrak{g }_{\mathrm{split}}$ associated to~$\Theta_{\mathrm{split}}$.

\begin{Theorem} \label{thm.nosplit}Let $\mathfrak{g}$ be a non-compact simple Lie algebra with simple system of restricted roots~$\Sigma $ and fix $\Theta\subset \Sigma $. Let $\mathbb{F}_{\Theta }$ be the flag manifold of $\mathfrak{g}$ determined by $\Theta $ and let $\mathbb{F}_{\Theta _{\mathrm{split}}}$ be as above. Then
\begin{gather*}
H_{2}(\mathbb{F}_\Theta,R)=H_{2}(\mathbb{F}_{\Theta _{\mathrm{split}}},R)\oplus \sum_{\alpha \in \Sigma _{2}^{\Theta }}R\cdot S_{r_{\alpha }}.
\end{gather*}
\end{Theorem}

\begin{Remark} The boundary operator $\partial _{\Theta }$ of the flag manifolds of a non-compact real simple Lie algebra $\mathfrak{g}$ is basically determined by the homology of its split part $\mathfrak{g}_{\mathrm{split}}$ as pointed out in \cite[p.~18]{RSM}. Theorem~\ref{thm.nosplit} is a particular instance of this fact.
\end{Remark}

\section{Homology groups of flag manifolds of split real forms}\label{sec.comput}

In this section we compute the second homology groups of real flag manifolds associated to \textit{split real simple Lie algebras}. We obtain these homology groups trough the explicit computation of the coefficients $c(w,w^{\prime })$. These coefficients are either zero or $\pm2$, so $\partial=0$ if the characteristic of $R$ is $\operatorname{char}R=2$. Thus in the sequel we assume $\operatorname{char}R\neq 2$.

The maximal flag case (with $\mathfrak{g}$ split) is treated first since it gives the model for the partial flag by restriction of the boundary map to the set of minimal cells. Even though the homology of flag manifolds of type $G_{2}$ is treated in~\cite{RSM}, we consider those here for the sake of completeness of the presentation.

\subsection{Maximal flag manifolds}

The results below make reference to the lines in Dynkin diagrams associated to simple real Lie algebras $\mathfrak{g}$ which are split.

\begin{Lemma}\label{lm.boundary2}Let $\alpha $, $\beta $ be simple roots with $\alpha \neq \beta $. Then $\partial S_{r_{\alpha }r_{\beta }}=\pm 2S_{r_{\beta }}$ if and only if~$\alpha $ and~$\beta $ are either simple linked, or double linked and~$\alpha $ is a long root.
\end{Lemma}

\begin{proof}By Proposition \ref{pro.boundarymax} we have $\partial S_{r_{\alpha}r_{\beta }}=c ( r_{\alpha }r_{\beta },r_{\beta } ) S_{r_{\beta }}$ where
\begin{gather}
c ( r_{\alpha }r_{\beta },r_{\beta } ) =\pm \big( 1-(-1) ^{\langle \alpha ^{\vee },\beta \rangle \dim \mathfrak{g}_{\beta}}\big). \label{eq.cabb}
\end{gather}
Since $\mathfrak{g}$ is a split real form $\dim \mathfrak{g}_{\beta }=1$, thus $c(r_{\alpha }r_{\beta },r_{\beta })=0$ if and only if $\langle \alpha^{\vee },\beta \rangle $ is even, otherwise $c(r_{\alpha }r_{\beta},r_{\beta })=\pm 2$ and $\partial S_{r_{\alpha }r_{\beta }}=\pm 2S_{r_{\beta }}$. We have $\langle \alpha ^{\vee },\beta \rangle $ is even only when $ \langle \alpha ^{\vee },\beta \rangle =0$ or when $\alpha $ and $\beta $ are linked by a double line with $\alpha $ the short root.
\end{proof}

\begin{Corollary}\label{cor.image1}If a diagram has only simple lines then $\partial \colon \mathcal{C}_{2}\longrightarrow \mathcal{C}_{1}$ is surjective. If a diagram has double lines then the image of $\partial \colon \mathcal{C}_{2}\longrightarrow \mathcal{C}_{1}$ is spanned $S_{r_{\beta }}$ with $\beta $ a short simple root, or $\beta $ long simple root such that there exists $\alpha \in \Sigma$ simple linked to $\beta $.
\end{Corollary}

Three simple roots $\alpha$, $\beta$, $\gamma$ are said to be in an $A_3$ configuration (in this order) if they are linked as follows

\begin{picture}(100,25)
\put(40,15){\circle{6}}
\put(37,3){$\alpha $}
\put(43,15){\line(1,0){20}}
\put(66,15){\circle{6}}
\put(63,3){$\beta $}
\put(69,15){\line(1,0){20}}
\put(92,15){\circle{6}}
\put(89,3){$\gamma $}
\end{picture}

In this case, $\partial S_{r_{\alpha }r_{\beta }}=\pm 2S_{r_{\beta }}$ and $\partial S_{r_{\gamma }r_{\beta }}=\pm 2S_{r_{\beta }}$, so there is a (unique) choice of a sign $\eta _{\alpha,\beta,\gamma }=\pm 1$ such that $\partial (S_{r_{\alpha }r_{\beta }}+\eta _{\alpha,\beta,\gamma}S_{r_{\gamma }r_{\beta }})=0$. Therefore, each $A_{3}$ configuration gives an element in the kernel of $\partial\colon \mathcal{C}_{2}\longrightarrow \mathcal{C}_{1}$.

\begin{Proposition}\label{pro.single}If a diagram has only simple lines the kernel of $\partial =\partial _{2}\colon \mathcal{C}_{2}\longrightarrow \mathcal{C}_{1}$ is spanned by the following elements:
\begin{itemize}\itemsep=0pt
\item $S_{r_{\alpha }r_{\beta }}$ with $ \langle \alpha,\beta \rangle =0$ and

\item $S_{r_{\alpha }r_{\beta }}+\,\eta _{\alpha,\beta,\gamma }S_{r_{\gamma }r_{\beta }}$ with $\alpha$, $\beta$, $\gamma $ in an $A_{3}$ configuration (in this order) and $\eta _{\alpha,\beta,\gamma }=\pm 1$ the
sign such that $\partial ( S_{r_{\alpha }r_{\beta }}+\eta _{\alpha,\beta,\gamma }S_{r_{\gamma }r_{\beta }} ) =0$.
\end{itemize}
\end{Proposition}

\begin{proof} Lemma~\ref{lm.boundary2} and the reasoning above show that the elements in the statement are indeed in the kernel of $\partial \colon \mathcal{C}_{2}\longrightarrow \mathcal{C}_{1}$. One should see that these are the only generators.

The boundary $\partial S_{r_{\alpha }r_{\beta }}$ is a multiple of $S_{r_{\beta }}$ hence an element of $\ker \partial _{2}$ is a sum of linear combinations of $2$-cells of the form
\begin{gather*}
\sum_{j\in J}n_{j}\,S_{r_{\alpha _{j}}r_{\beta }},\qquad\mbox{with} \quad n_{j}\in R \quad \mbox{for all} \ j\in J.
\end{gather*}
In such a linear combination we can take $\alpha _{j}$ not orthogonal to $\beta $ for all $j$ because $\partial S_{r_{\alpha }r_{\beta }}=0$ if $\langle \alpha,\beta \rangle =0$.

In a simply laced diagram a set of roots $\{\alpha_{j},\beta \}$ appearing in a linear combination as above $\langle \alpha _{j},\beta \rangle \neq 0$ occurs only in an $A_{3}$ configuration (with $\beta $ the middle root) or if the roots are in a~$D_{4}$ configuration as follows

\vspace{14mm}

\begin{picture}(100,25)

\put(66,41){\circle{6}}
\put(63,53){$\alpha_2 $}
\put(66,18){\line(0,1){20}}
\put(40,15){\circle{6}}
\put(37,3){$\alpha_1 $}
\put(43,15){\line(1,0){20}}
\put(66,15){\circle{6}}
\put(63,3){$\beta $}
\put(69,15){\line(1,0){20}}
\put(92,15){\circle{6}}
\put(89,3){$\alpha_3 $}

\end{picture}

A linear combination $n_{1}S_{r_{\alpha _{1}}r_{\beta}}+n_{2}S_{r_{\alpha _{2}}r_{\beta }}+n_{3}S_{r_{\alpha _{3}}r_{\beta }}\in \ker \partial $ associated to a $D_{4}$ configuration belongs to the span of the combinations $S_{r_{\alpha }r_{\beta }}+\eta _{\alpha,\beta,\gamma}S_{r_{\gamma }r_{\beta }}$ given by $A_{3}$ configurations where as above $\eta _{\alpha,\beta,\gamma }=\pm 1$ is the sign such that $\partial (S_{r_{\alpha }r_{\beta }}+\eta _{\alpha,\beta,\gamma }S_{r_{\gamma}r_{\beta }}) =0$. In fact, suppose first that $\partial S_{r_{\alpha _{i}}r_{\beta}}=\partial S_{r_{\alpha _{j}}r_{\beta }}$ for all $i\neq j$, that is, $\eta_{\alpha _{1},\beta,\alpha _{2}}=\eta _{\alpha _{1},\beta,\alpha_{3}}=\eta _{\alpha _{2},\beta,\alpha _{3}}=-1$. Then
\begin{gather*}
\partial \big( n_{1}S_{r_{\alpha _{1}}r_{\beta }}+n_{2}S_{r_{\alpha_{2}}r_{\beta }}+n_{3}S_{r_{\alpha _{3}}r_{\beta }}\big) =0
\end{gather*}
is the same as $2 ( n_{1}+n_{2}+n_{3} ) =0$ so that $n_{3}=-(n_{1}+n_{2})$ because $\operatorname{char}R\neq 2$. Hence
\begin{align*}
\sum_{i=1}^{3}n_{i}S_{r_{\alpha _{i}}r_{\beta }} &= n_{1}(S_{r_{\alpha_{1}}r_{\beta }}-S_{r_{\alpha _{3}}r_{\beta }})+n_{2}(S_{r_{\alpha_{2}}r_{\beta }}-S_{r_{\alpha _{3}}r_{\beta }}) \\
&=n_{1}(S_{r_{\alpha _{1}}r_{\beta }}+\eta _{\alpha _{1},\beta,\alpha_{3}}S_{r_{\alpha _{3}}r_{\beta }})+n_{2}(S_{r_{\alpha _{2}}r_{\beta }}+\eta_{\alpha _{2},\beta,\alpha _{3}}S_{r_{\alpha _{3}}r_{\beta }})
\end{align*}
as we wanted to show. If the images $\partial S_{r_{\alpha _{1}}r_{\beta }}$ have two coincident signs and one opposite, a~similar argument gives $\sum\limits_{i=1}^{3}n_{i}S_{r_{\alpha _{i}}r_{\beta }}$ as a combination of the elements in the kernel given by the~$A_{3}$ subdiagrams.
\end{proof}

In a diagram with double lines, another configuration becomes relevant. The roots $\alpha$, $\beta$, $\gamma$ are said to be in a $C_3$ configuration if they are linked as follows

\begin{picture}(100,25)
\put(40,15){\circle{6}}
\put(37,3){$\alpha $}
\put(43,15){\line(1,0){20}}
\put(66,15){\circle{6}}
\put(63,3){$\beta $}
\put(69,15){\line(1,2){5}}
\put(69,15){\line(1,-2){5}}
\put(69,16.2){\line(1,0){20}}
\put(69,13.8){\line(1,0){20}}
\put(92,15){\circle{6}}
\put(89,3){$\gamma $}
\end{picture}

As in an $A_{3}$ configuration, $\partial S_{r_{\alpha }r_{\beta }}=\pm 2S_{r_{\beta }}$ and $\partial S_{r_{\gamma }r_{\beta }}=\pm 2S_{r_{\beta }}$ because $\beta $ is the short root in the double link with $\gamma $. Hence there is a (unique) sign $\eta_{\alpha,\beta,\gamma }=\pm 1$ such that $\partial (S_{r_{\alpha }r_{\beta }}+\eta _{\alpha,\beta,\gamma}S_{r_{\gamma }r_{\beta }})=0$.

\begin{Proposition}\label{pro.double}If a diagram has double lines, then the kernel of $\partial \colon \mathcal{C}_{2}\longrightarrow \mathcal{C}_{1}$ is spanned by~$S_{r_{\alpha }r_{\beta }}$ with $\langle \alpha,\beta \rangle =0$ and $S_{r_{\alpha }r_{\beta }}+\eta _{\alpha,\beta,\gamma }S_{r_{\gamma }r_{\beta }}$ with $\alpha$, $\beta$, $\gamma $ in an $A_{3}$ configuration, as in the previous proposition together with the following generators:
\begin{itemize}\itemsep=0pt
\item $S_{r_{\alpha }r_{\beta }}$ with $\alpha \neq \beta $ double linked and $\alpha $ the short root and
\item $S_{r_{\alpha }r_{\beta }}+\,\mu _{\alpha,\beta,\gamma }S_{r_{\gamma}r_{\beta }}$ with $\alpha$, $\beta$, $\gamma $ in a $C_{3}$ configuration $($in this order$)$ and $\mu _{\alpha,\beta,\gamma }=\pm 1$ the sign such that $\partial ( S_{r_{\alpha }r_{\beta }}+\eta _{\alpha,\beta,\gamma}S_{r_{\gamma }r_{\beta }} ) =0$.
\end{itemize}
\end{Proposition}

\begin{proof} The proof of this proposition is the same as the previous one. The elements in the previous proposition and those in the statement belong to the kernel of $\partial \colon \mathcal{C}_{2}\longrightarrow \mathcal{C}_{1}$.

On the other hand for a linear combination of $2$-dimensional cells in $\ker \partial $ we can assume that it has the form $\sum\limits_{j\in J}n_{j}S_{r_{\alpha _{j}}r_{\beta }}$ with $\alpha _{j}$ not orthogonal to $\beta $ for all $j$. Again, the only possibilities is that $\beta $ and $\alpha _{i}^{\prime }s$ are in a $C_{3}$, $A_{3}$ or a $D_{4}$ configuration. The last case is in fact a linear combination of $A_{3}$ configurations so the result follows.
\end{proof}

With the next proposition complete the computation of the kernel of $\partial \colon \mathcal{C}_{2}\longrightarrow \mathcal{C}_{1}$ by considering the Lie algebra $G_{2}$.

\begin{Proposition}\label{pro.kerg2}If the diagram is of type $G_{2}$ then the kernel of $\partial \colon \mathcal{C}_{2}\longrightarrow \mathcal{C}_{1}$ is zero.
\end{Proposition}

\begin{proof}Recall that $G_{2}$ has two simple roots $\alpha _{1},\alpha _{2}$ linked by a triple line. Thus $S_{r_{\alpha _{1}}r_{\alpha _{2}}}$ and $S_{r_{\alpha_{2}}r_{\alpha _{1}}}$ span $\mathcal{C}_{2}$. These roots satisfy $\langle \alpha _{1}^{\vee },\alpha _{2}\rangle =-1$ and $\langle \alpha _{2}^{\vee },\alpha _{1}\rangle =-3$. Then by (\ref{forsigmasoma}), $\sigma (r_{\alpha _{i}}r_{\alpha _{j}},r_{\alpha_{j}})=\langle \alpha _{i}^{\vee },\alpha _{j}\rangle $ is odd for $1\leq i\neq j\leq 2$and thus $c(r_{\alpha _{i}}r_{\alpha_{j}},r_{\alpha _{j}})=\pm 2$. Therefore $\partial S_{r_{\alpha _{1}}r_{\alpha _{2}}}=\pm 2S_{r_{\alpha _{2}}}$ and $\partial S_{r_{\alpha _{2}}r_{\alpha _{1}}}=\pm 2S_{r_{\alpha _{1}}}$.
\end{proof}

Next we focus on the computation of the boundaries of the $3$-dimensional cells, that is, the image of $\partial\colon \mathcal{C}_{3}\longrightarrow \mathcal{C}_{2}$. Such cells are of the form $S_{r_{\alpha }r_{\beta
}r_{\gamma }}$ with $\alpha \neq \beta \neq \gamma $ ($\alpha =\gamma $ is allowed) whose boundaries are given by Proposition~\ref{pro.boundarymax}:
\begin{gather*}
\partial S_{r_{\alpha }r_{\beta }r_{\gamma }}=c ( r_{\alpha }r_{\beta}r_{\gamma },r_{\alpha }r_{\gamma } ) S_{r_{\alpha }r_{\gamma}}+c ( r_{\alpha }r_{\beta }r_{\gamma },r_{\beta }r_{\gamma } )S_{r_{\beta }r_{\gamma }}
\end{gather*}
with $c ( r_{\alpha }r_{\beta }r_{\gamma },r_{\alpha }r_{\gamma } ) =0$ in the case $\gamma =\alpha $. The coefficients have the form $\pm ( 1-(-1) ^{\sigma } ) $, where~$\sigma $ is as in~(\ref{forsigmasoma}). Namely,
\begin{itemize}\itemsep=0pt
\item $\sigma ( r_{\alpha }r_{\beta }r_{\gamma },r_{\alpha }r_{\gamma} ) =\langle \beta ^{\vee },\gamma \rangle $, for $\gamma \neq \alpha $, since the only positive root taken to negative by $r_{\gamma }$ is $\gamma
$ itself. Thus
\begin{gather}
c ( r_{\alpha }r_{\beta }r_{\gamma },r_{\alpha }r_{\gamma } ) =\pm \big( 1-(-1) ^{\langle \beta ^{\vee },\gamma \rangle }\big).\label{eq.cabcac}
\end{gather}
\item $\sigma ( r_{\alpha }r_{\beta }r_{\gamma },r_{\beta }r_{\gamma} ) =\langle \alpha ^{\vee },\beta \rangle +\langle \alpha ^{\vee},r_{\beta }\gamma \rangle $ since the positive roots taken to negative by $r_{\gamma }r_{\beta }= ( r_{\beta }r_{\gamma } ) ^{-1}$ are $\beta $ and $r_{\beta }\gamma $. Therefore,
\begin{gather} \label{eq.cabcbc}
c ( r_{\alpha }r_{\beta }r_{\gamma },r_{\beta }r_{\gamma } ) =\pm \big( 1-(-1) ^{\langle \alpha ^{\vee },\beta \rangle +\langle \alpha ^{\vee },r_{\beta }\gamma \rangle }\big).
\end{gather}
\end{itemize}

\begin{Proposition}\label{pro.bound3bab}\label{pro.bound3ort}Let $\alpha,\beta,\gamma \in \Sigma $.
\begin{enumerate}\itemsep=0pt
\item[$1.$] If $\alpha$ and $\beta$ are linked by a double line and $\alpha$ is a short root then
\begin{gather*}
\partial S_{ r_{\beta }r_{\alpha }r_{\beta }} =\pm 2S_{ r_{\alpha }r_{\beta}}.
\end{gather*}
\item[$2.$] If $\langle\alpha,\beta\rangle=0$, $\beta$ is long and $\gamma$ is such that $ \langle \gamma^\vee,\alpha \rangle=-1$ then
\begin{gather*}
\partial S_{ r_{\gamma}r_{\alpha }r_{\beta }} =\pm 2S_{ r_{\alpha }r_{\beta}}.
\end{gather*}
\item[$3.$] If $\langle \alpha,\beta \rangle =0$ and $\langle \gamma ^{\vee},\beta \rangle =-1$ then
\begin{gather*}
\partial S_{ r_{\alpha }r_{\gamma }r_{\beta }} =\pm 2S_{ r_{\alpha }r_{\beta}}.
\end{gather*}
\end{enumerate}
\end{Proposition}

\begin{proof}If $\alpha $ and $\beta $ are as in (1) then $\langle \beta ^{\vee},\alpha \rangle =-1$ and $\langle \alpha ^{\vee },\beta\rangle =-2$. Hence
\begin{align*}
\sigma ( r_{\beta }r_{\alpha }r_{\beta },r_{\alpha }r_{\beta } )&=\langle \beta ^{\vee },\alpha \rangle +\langle \beta ^{\vee },r_{\alpha}\beta \rangle =\langle \beta ^{\vee },\alpha \rangle +\langle \beta ^{\vee },\beta
+2\alpha \rangle \\
&=3\langle \beta ^{\vee },\alpha \rangle +\langle \beta ^{\vee },\beta\rangle =-1,
\end{align*}
which is odd. Therefore $c ( r_{\beta }r_{\alpha }r_{\beta },r_{\alpha}r_{\beta } ) =\pm 2$ showing the first assertion.

For the second item, $ \langle \alpha ^{\vee },\beta \rangle =0$ implies $\partial S_{r_{\gamma }r_{\alpha }r_{\beta }}=c(r_{\gamma}r_{\alpha }r_{\beta },r_{\alpha }r_{\beta })S_{r_{\alpha }r_{\beta }}$ which is non-zero since
\begin{gather*}
\sigma (r_{\gamma }r_{\alpha }r_{\beta },r_{\alpha }r_{\beta })= \langle \gamma ^{\vee },\alpha \rangle + \langle \gamma ^{\vee },\beta \rangle =-1+ \langle \gamma ^{\vee },\beta \rangle
\end{gather*}
is odd.

For the last assertion, the assumption $\langle \gamma ^{\vee },\beta \rangle =-1$ implies that removing from $r_{\alpha }r_{\gamma }r_{\beta }$ the reflection $r_{\gamma }$ one has
\begin{gather*}
c ( r_{\alpha }r_{\gamma }r_{\beta },r_{\alpha }r_{\beta } ) =\pm 2
\end{gather*}%
by (\ref{forsigmasoma}). On the other hand, the exponent reached when removing the first reflection $r_{\alpha }$ is
\begin{gather*}
\sigma ( r_{\alpha }r_{\gamma }r_{\beta },r_{\gamma }r_{\beta } ) =\langle \alpha ^{\vee },\gamma +r_{\gamma }\beta \rangle =\langle \alpha,2\gamma +\beta \rangle =2\langle \alpha,\gamma \rangle,
\end{gather*}
which is even. Hence $c ( r_{\alpha }r_{\gamma }r_{\beta },r_{\gamma}r_{\beta } ) =0$ and $\partial S_{r_{\alpha }r_{\gamma }r_{\beta}}=\pm 2S_{r_{\alpha }r_{\beta }}$.
\end{proof}

\begin{Proposition}\label{pro.bound3A3}Let $\alpha,\beta,\gamma \in \Sigma $ in an $A_{3}$ or a $C_{3}$ configuration. Then
\begin{gather*}
\partial S_{r_{\alpha }r_{\gamma }r_{\beta }}=\pm 2\big( S_{r_{\alpha}r_{\beta }}+\eta _{\alpha,\beta,\gamma }S_{r_{\gamma }r_{\beta }}\big),
\end{gather*}
where $\eta _{\alpha,\beta,\gamma }$ is the sign such that $\partial ( S_{r_{\alpha }r_{\beta }}+\eta _{\alpha,\beta,\gamma }S_{r_{\gamma}r_{\beta }}) =0$.
\end{Proposition}

\begin{proof}By hypothesis, the roots verify $\langle \gamma ^{\vee },\beta \rangle =-1= \langle \alpha ^{\vee },\beta \rangle $ and $ \langle \alpha,\gamma \rangle =0$. These conditions give $c( r_{\alpha}r_{\gamma }r_{\beta },r_{\alpha }r_{\beta }) =\pm 2$. In addition, $\sigma (r_{\alpha }r_{\gamma }r_{\beta },r_{\gamma }r_{\beta })=\langle\alpha ^{\vee },\gamma +r_{\gamma }\beta \rangle =\langle \alpha^{\vee },\beta \rangle =-1$ so $c(r_{\alpha }r_{\gamma }r_{\beta},r_{\gamma }r_{\beta })=\pm 2$. This proves $\partial S_{r_{\alpha}r_{\gamma }r_{\beta }}=\pm 2(S_{r_{\alpha }r_{\beta }}+\nu S_{r_{\gamma}r_{\beta }})$, for some $\nu =\pm 1$. The fact that $\partial ^{2}=0$ implies $\partial ( S_{r_{\alpha }r_{\beta }}+\nu S_{r_{\gamma }r_{\beta }}) =0$ which forces $\nu =\eta_{\alpha,\beta,\gamma }$, and the choice of the sign is the right one.
\end{proof}

The results in this section yield the following expressions for the $2$-homology of the maximal flag manifolds.

\begin{Theorem}\label{thm.2homfull}Let $\mathbb{F}$ be the maximal flag manifold of a split real form~$\mathfrak{g}$. Then
\begin{alignat*}{3}
& H_{2} ( \mathbb{F},R ) = 0\qquad && \mbox{if }\mathfrak{g} \mbox{ is of type }G_{2}\mbox{ and} & \\
& H_{2} ( \mathbb{F},R ) = R/2R\oplus \cdots \oplus R/2R \qquad && \mbox{otherwise}.&
\end{alignat*}
The number of summands equals the number of generators of $\ker( \partial\colon \mathcal{C}_{2}\longrightarrow \mathcal{C}_{1})$ given in Propositions~{\rm \ref{pro.single}} and~{\rm \ref{pro.double}}.
\end{Theorem}

\begin{proof}If $\mathfrak{g}$ is of type $G_{2}$ the kernel of $\partial $ is trivial as shown in Proposition~\ref{pro.kerg2}. If $\mathfrak{g}$ is not of type $G_{2} $, the generators of the kernel of $\partial\colon \mathcal{C}_{2}\longrightarrow \mathcal{C}_{1}$ were given in Propositions~\ref{pro.single} and~\ref{pro.double}. Each such generator has a $\pm 2$ multiple which is indeed a boundary. In fact, let $S_{r_{\alpha }r_{\beta }}$ be a $2$-cell such that $\partial S_{r_{\alpha }r_{\beta }}=0$. Then either $\langle \alpha,\beta \rangle =0$ or $\alpha $ and $\beta $ are double linked with $\alpha $ short.

In the double linked case with $\alpha $ the short root we have $\partial S_{r_{\beta }r_{\alpha }r_{\beta }}=\pm 2S_{r_{\alpha }r_{\beta }}$ because of~(1) in Proposition~\ref{pro.bound3bab}.

On the other hand for $ \langle \alpha,\beta \rangle =0$ suppose first that $\mathfrak{g}$ is of type $C_{l}$ and $\beta =\alpha _{l}$. Then there exists $\gamma \in \Sigma $ such that $ \langle \gamma ^{\vee},\alpha \rangle =-1$ so that (2) of Proposition~\ref{pro.bound3bab} gives $\partial S_{r_{\gamma }r_{\alpha }r_{\beta }}=\pm 2S_{r_{\alpha}r_{\beta }}$. If~$\mathfrak{g}$ is not of type $C_{l}$ or $\beta \neq \alpha _{l}$ in the $C_{l}$ case there always exists $\gamma \in \Sigma $ with $\langle \gamma ^{\vee },\beta \rangle =-1$. Item~(3) in the same proposition gives $\partial S_{r_{\alpha }r_{\gamma }r_{\beta }}=\pm 2S_{r_{\alpha }r_{\beta }}$.

Finally, each generator of the form $S_{r_{\alpha }r_{\gamma }r_{\beta }}$ with $\alpha,\beta,\gamma \in \Sigma $ in an $A_{3}$ or $C_{3}$ diagram has a $\pm 2$ multiple in the image of $\partial\colon \mathcal{C}_{3}\longrightarrow \mathcal{C}_{2}$ as shown in Proposition~\ref{pro.bound3A3}.
\end{proof}

\subsection{Partial flag manifolds}

At this point, all the relevant computations of kernel and images of the boundary map is done. To deal with the homology of the partial case, $\Theta \neq \varnothing$, we need to determine when does a cell~$S_{w}$ in the kernel corresponds to a minimal element and when does a boundary in~$R\cdot S_{w}$ is the image of a minimal cell.

\begin{Proposition}If a diagram has only single and double lines, the kernel of $\partial_\Theta^{\min}\colon \mathcal{C}_2^\Theta\longrightarrow \mathcal{C}_1^\Theta$ is spanned by the following elements:
\begin{itemize}\itemsep=0pt
\item $S_{r_\alpha r_\beta}$ with $\alpha,\beta\notin\Theta$ and $\langle\alpha, \beta \rangle=0$,
\item $S_{r_\alpha r_\beta}$ with $\beta\notin \Theta$, $\alpha$ and $\beta$ double linked and $\alpha$ short,
\item $S_{r_\alpha r_\beta}+ \eta^*_{\alpha,\beta,\gamma} S_{r_\gamma r_\beta}$ whenever $\alpha$, $\beta$, $\gamma$ form a $*_3$ diagram, in this order, with $\beta\notin \Theta$.
\end{itemize}
\end{Proposition}

\begin{proof}By Lemma \ref{lm.minimalroots}, any $w=r_{\alpha }r_{\beta }$ is a minimal element under the first two conditions, and also $r_{\alpha }r_{\beta }$ and $r_{\gamma }r_{\beta }$ are minimal when they fit into an $A_{3}$ or $C_{3}$ diagram with $\beta \notin \Theta $. Hence any element in the list above is indeed in $\mathcal{C}_{2}^{\Theta }$. Also, their image under $\partial $ in the maximal flag is zero, so $\partial _{\Theta }^{\min }$ is zero.
\end{proof}

In the case of a $G_2$ diagram, $\ker (\partial\colon \mathcal{C}_2\longrightarrow\mathcal{C}_1)=\{0\}$ and thus $\ker (\partial_\Theta^{\min}\colon \mathcal{C}_2\longrightarrow\mathcal{C}_1)$ is zero independently of $\Theta$.

By Lemma~\ref{lm.minimalroots}, $w= r_\alpha r_\beta r_\gamma$ is a minimal element in $\mathcal{W}$ if and only if $\gamma$, $r_\gamma\beta$, $r_\gamma r_\beta\alpha$ are all
outside $ \langle \Theta \rangle$.

\begin{Proposition}\label{pro.minimal3}For any $\beta \notin \Theta $, the following cells are in $\mathcal{C}_{3}^{\Theta }$:
\begin{itemize}\itemsep=0pt
\item $S_{r_\beta r_\alpha r_\beta}$ with $\alpha, \beta$ double linked and $\alpha$ short;

\item $S_{ r_{\alpha }r_{\gamma }r_{\beta }}$ with $\alpha$, $\beta$, $\gamma$ in an $A_{3}$ or $C_3$ diagram $($in this order$)$;

\item $S_{ r_{\alpha }r_{\gamma }r_{\beta }}$ with $\alpha\notin \Theta$, $\beta\neq \alpha$ and $ \langle \gamma, \beta \rangle\neq 0$;

\item $S_{r_{\gamma }r_{\alpha }r_{\beta }}$ with $\alpha\notin \Theta$ and $ \langle \alpha, \gamma \rangle\neq 0$.
\end{itemize}
\end{Proposition}

\begin{proof}Let $w= r_\beta r_\alpha r_\beta$ with $\alpha$ and $\beta$ double linked, $\alpha$ short and $\beta\notin\Theta$. Then $\beta$, $r_\beta\alpha=\alpha+\beta $ and $r_\beta r_\alpha\beta=\beta+2\alpha$ are all outside $\langle \Theta\rangle$. Hence $S_{r_\beta r_\alpha r_\beta}\in \mathcal{C}_3^\Theta$.

Let $\alpha$, $\beta$, $\gamma$ be in $*_3$ configuration (in this order) and let $w= r_{\alpha }r_{\gamma }r_{\beta }$. Then $r_\beta\gamma=\gamma+\langle \beta^\vee,\gamma\rangle\beta$ and $r_\beta r_\gamma\alpha=r_\beta\alpha=\beta+\alpha$ have non-zero component in~$\beta$, so~$w$ is minimal.

If $\beta\neq \alpha$, $\alpha,\beta\notin\Theta$ and $\langle\gamma,\beta\rangle\neq 0$ then $r_\beta\gamma$ has non-zero component in $\beta$ and $r_\beta r_\gamma\alpha=\alpha-\langle\gamma^\vee,\alpha\rangle-\langle
\beta^\vee,r_\gamma\alpha\rangle\beta$ which has non-zero component in~$\alpha$. Therefore $w=r_\alpha r_\gamma r_\beta $ is minimal. The proof of the last item follows in a similar way.
\end{proof}

The propositions above together with the proof of Theorem~\ref{thm.2homfull} show that if $S_w$ is a $2$-di\-men\-sional cell corresponding $w\in \mathcal W_\Theta^{\rm min}$ then $\pm 2 S_w$ is a border. This leads to the following conclusion.

\begin{Theorem}\label{thm.2homint}Let $\mathbb{F}_{\Theta }$ be a partial flag manifold of a split real form~$\mathfrak{g}$. Then for any $\Theta $
\begin{alignat*}{3}
& H_{2} ( \mathbb{F}_{\Theta },R ) = 0 \qquad && \mbox{if }\mathfrak{g} \mbox{ is of type }G_{2}\mbox{ and}& \\
& H_{2} ( \mathbb{F}_{\Theta },R ) = R/2R\oplus \cdots \oplus R/2R \qquad && \mbox{otherwise}. &
\end{alignat*}
The number of summands equals $\dim \ker \big(\partial _{\Theta }^{\min }\colon \mathcal{C}_{2}^{\Theta }\longrightarrow \mathcal{C}_{1}^{\Theta }\big)$.
\end{Theorem}

\section{Classifications}

Let $\mathbb{F}_{\Theta }$ be a flag manifold associated to a non-compact simple real Lie algebra $\mathfrak{g}$. Theorems~\ref{thm.nosplit},~\ref{thm.2homfull} and \ref{thm.2homint} imply that the second homology group of $\mathbb{F}_{\Theta }$ is determined by the simple roots in~$\Sigma _{\mathrm{split}}$ and $\Sigma _{2}$.

The multiplicities (and hence the ranks) of the restricted roots of simple real Lie algebras can be read from the classification table of the real forms (see, e.g., \cite[pp.~30--32]{WarG}). If for a~Lie algebra $\mathfrak{g}$ the sets $\Sigma _{\mathrm{split}}$ and $\Sigma _{2}$ are empty (that is, if $\operatorname{rank}\alpha \geq 3$ for every $\alpha \in \Sigma $) then the second homology of any flag manifold associated to $\mathfrak{g}$ is zero. The Lie algebras having roots with $\operatorname{rank} \alpha \leq 2$ are listed in Table \ref{table} below. The homology of the flag manifolds associated to the Lie algebras in the list is given in the following theorem that summarizes Theorems~\ref{thm.nosplit},~\ref{thm.2homfull} and~\ref{thm.2homint}.

\begin{Theorem}\label{teoResumo}Let $\mathbb{F}_{\Theta }$ be the flag manifold associated to a non-compact simple real Lie algebra~$\mathfrak{g}$ and to the subset $\Theta \subset \Sigma $ of restricted simple roots. If $\operatorname{char}R\neq 2$ then
\begin{itemize}\itemsep=0pt
\item $H_{2}(\mathbb{F}_{\Theta },R)=0$ if and only if $\mathfrak{g}$ is of type $G_{2}$ for any $\Theta $ or if $\mathfrak{g}$ is not of type $G_{2}$ and both $\Sigma _{\mathrm{split}}$ and $\Sigma _{2}$ are contained in $\Theta $. This is the case if $\mathfrak{g}$ does not appear in Table~{\rm \ref{table}}.

\item $H_{2}(\mathbb{F}_{\Theta },R)$ is non-zero and has only torsion components $R/2R$ if and only if $\mathfrak{g}$ is not of type $G_{2}$, $\Sigma _{2}\subset \Theta $ and $\Sigma _{\mathrm{split}}$ is not contained
in $\Theta $. In particular, if $\mathfrak{g}$ is a split real form, not of type~$G_{2}$.

\item $H_{2}(\mathbb{F}_{\Theta },R)$ contains a free $R$ module if and only if $\Sigma _{2}\cap (\Sigma \backslash \Theta )\neq \varnothing $. The rank of the module equals the cardinality of this intersection.
\end{itemize}
\end{Theorem}

\noindent
\textbf{Notation.} In Table \ref{table}, the simple roots of $B_{l}$, $C_{l}$ and $F_{4}$ are labelled according to the following diagrams

\centerline{\includegraphics{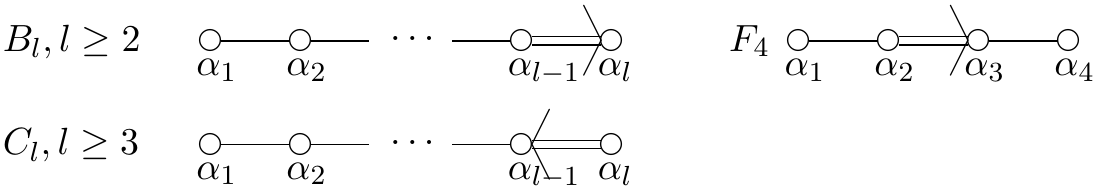}}

\begin{table}[h]\centering
\begin{tabular}{|l|c|c|c|}
\hline
Lie algebra & Dynkin diagram & $\Sigma_{\mathrm{split}}$ & $\Sigma_2 $ \\
\hline
$\mathrm{AI}$ & $A_l $ & $\Sigma $ & --- \\ \hline
$\mathrm{AIII_1}$ & $B_l $ & --- & $\{\alpha_{1}, \ldots, \alpha_{l-1} \} $
\\ \hline
$\mathrm{AIII_2}$ & $C_l $ & $\{\alpha_{l} \} $ & $\{\alpha_1, \ldots,
\alpha_{l-1}\}$ \\ \hline
$\mathrm{BI_1} $ & $B_l $ & $\Sigma $ & --- \\ \hline
$\mathrm{BI_2}$ & $B_l $ & $\{\alpha_1,\ldots,\alpha_{l-1}\}$ & --- \\ \hline
$\mathrm{CI}$ & $C_l $ & $\Sigma$ & --- \\ \hline
$\mathrm{DI_1}$ & $B_l $ & $\{\alpha_1, \dots, \alpha_{l-1}\}$ & --- \\ \hline
$\mathrm{DI_2}$ & $B_l $ & $\{\alpha_1, \ldots, \alpha_{l-1}\}$ & $%
\{\alpha_l\}$ \\ \hline
$\mathrm{DI_3}$ & $D_l $ & $\Sigma$ & --- \\ \hline
$\mathrm{DIII_1}$ & $C_l $ & $\{\alpha_l\}$ & --- \\ \hline
$\mathrm{EI}$ & $E_6 $ & $\Sigma$ & --- \\ \hline
$\mathrm{EII}$ & $F_4 $ & $\{\alpha_1,\alpha_2\}$ & $\{\alpha_3,\alpha_4\}$
\\ \hline
$\mathrm{EV}$ & $E_7 $ & $\Sigma$ & --- \\ \hline
$\mathrm{EVI}$ & $F_4 $ & $\{\alpha_1,\alpha_2\}$ & --- \\ \hline
$\mathrm{EVII}$ & $C_3 $ & $\{\alpha_1\}$ & \\ \hline
$\mathrm{EVIII}$ & $E_8 $ & $\Sigma$ & --- \\ \hline
$\mathrm{EIX}$ & $F_4 $ & $\{\alpha_1,\alpha_2\}$ & --- \\ \hline
$\mathrm{FI}$ & $F_4 $ & $\Sigma$ & --- \\ \hline
$\mathrm{G}$ & $G_2 $ & $\Sigma$ & --- \\ \hline
\end{tabular}
\caption{Lie algebras having roots with rank $\leq 2 $.}\label{table}
\end{table}

\section{De Rham cohomology}

The computation of the second de Rham cohomology group of real flag manifolds was the first motivation of this work. This cohomology group can be obtained from the homology groups studied along this paper. In fact, by the universal coefficient theorem (see for instance \cite{Ha}) we have, for a real flag manifold $\mathbb{F}_{\Theta }$, $H_{\mathrm{dR}}^{2}(\mathbb{F}_{\Theta },\mathbb{R})\simeq \mathbb{R}^{s}$ where $s$ is the rank of the free $\mathbb{Z}$-module in $H_{2}(\mathbb{F}_{\Theta },\mathbb{Z})$. Therefore the description of the homologies of the flag manifolds of Theorem~\ref{teoResumo} combined with Table~\ref{table} yield the following.

\begin{Theorem}\label{tedeRham}Let $\mathbb{F}_{\Theta }$ be the flag manifold associated to a non-compact simple real Lie algebra~$\mathfrak{g}$ and to the subset $\Theta \subset \Sigma $ of restricted simple roots. Then $H_{\mathrm{dR}}^{2}(\mathbb{F}_{\Theta },\mathbb{R})=0$ unless $\mathfrak{g}$ is of types $\mathrm{AIII_{1}}$, $\mathrm{AIII_{2}}$, $\mathrm{DI_{2}}$, $\mathrm{EII}$ and $\Sigma _{2}\cap (\Sigma \backslash \Theta )\neq \varnothing $. In that case $H_{\mathrm{dR}}^{2}(\mathbb{F}_{\Theta },\mathbb{R})\simeq \mathbb{R}^{s}$ with $s=\vert \Sigma _{2}\cap (\Sigma \backslash \Theta)\vert $.
\end{Theorem}

As particular cases, one obtains that $H_{\mathrm{dR}}^{2}(\mathbb{F}_{\Theta },\mathbb{R})=0$ for any flag manifold $\mathbb{F}_{\Theta }$ of a~split real form.

The objective now is to get invariant differential forms that represent the $2$-cohomologies of the flag manifolds of the Lie algebras $\mathrm{AIII_{1}}$, $\mathrm{AIII_{2}}$, $\mathrm{DI_{2}}$ and $\mathrm{EII}$. These $2$-forms will be obtained by the Weil construction that provides closed differential forms as characteristic forms of principal bundles. In the next theorem we recall this construction for differential $2$-forms

\begin{Theorem}[{see Kobayashi--Nomizu \cite[Chapter XII]{kn}}]\label{teoWeilhomo} Let $\pi\colon Q\rightarrow M$ be a principal bundle with structural group $L$ having Lie algebra $\mathfrak{l}$. Endow $Q$ with a connection form $\omega $ whose curvature $2$-form $($with values in $\mathfrak{l})$ is $\Omega $. Take $f\in \mathfrak{l}^{\ast }$ which is $L$-invariant $($that is, $f\circ \operatorname{Ad} (g) =f$ for all $g\in L)$. Then the $2$-form $f\circ \Omega $ is such that there exists a closed $2$-form $\widetilde{f}$ on $M$ with $f\circ \Omega =\pi ^{\ast }\widetilde{f}$. The de Rham cohomology class of $\widetilde{f}$ remains the same if the connection is changed.
\end{Theorem}

As a complement to this theorem we note that if $L$ is compact then its Lie algebra $\mathfrak{l}$ is reductive and a necessary condition for the existence of an invariant $f\in \mathfrak{l}^{\ast }$, $f\neq 0$, is that $\mathfrak{l}$ has non trivial center $\mathfrak{z} ( \mathfrak{l})$, that is, $\mathfrak{l}$ is not semi-simple and is $\neq \{0\}$. If furthermore $L$ is connected then this condition is also sufficient and an invariant $f$ is given by $f(\cdot) =\langle X,\cdot \rangle $ with $X\in \mathfrak{z} ( \mathfrak{l} ) $ where $\langle \cdot,\cdot \rangle $ is an invariant inner product in $\mathfrak{l}$. If $L$ is not connected the invariant $f\in \mathfrak{l}^{\ast }$ are given by $f ( \cdot ) =\langle X,\cdot \rangle $ as well with $X\in \mathfrak{z} ( \mathfrak{l} ) $ fixed by the non-identity components of~$L$. This construction yields a map $\mathfrak{z} ( \mathfrak{l} ) \rightarrow H^{2} ( M,\mathbb{R} ) $ that depends in an isomorphic way on the invariant inner product $\langle \cdot,\cdot \rangle $ in $\mathfrak{l}$. (This map can be defined also via a negative definite form like $-\langle \cdot,\cdot \rangle $.)

Given a flag manifold $\mathbb{F}_{\Theta }=K/K_{\Theta }$ we will apply the Weil construction to the principal bundle $K\rightarrow K/K_{\Theta }$. For this bundle in which the total space is the Lie group $K$ we can take left invariant connections yielding invariant differential forms in the base space $K/K_{\Theta }$.

For a flag manifold $\mathbb{F}_{\Theta }=K/K_{\Theta }$ associated to one of the Lie algebras $\mathrm{AIII}_{1}$, $\mathrm{AIII}_{2}$, $\mathrm{DI}_{2}$ or $\mathrm{EII}$ we intend to prove that there are enough $f\in \mathfrak{k}_{\Theta }^{\ast }$ so that the $2$-forms $\widetilde{f}$ exhaust the $2$-cohomology. To this purpose it is required to describe the center of $\mathfrak{k}_{\Theta }$. When $\mathbb{F}_{\Theta }=\mathbb{F}$ is a maximal flag manifold then $\mathbb{F}=K/M$ where $M$ is the centralizer of $\mathfrak{a}$ in $K$. In the next section we discuss the center $\mathfrak{z}(\mathfrak{m}) $ of the Lie algebra $\mathfrak{m}$ of $M$.

\section[$M$-group and Satake diagrams]{$\boldsymbol{M}$-group and Satake diagrams}\label{seccenterM}

In this section we obtain preparatory results that will allow us, in the next section, to get the characteristic forms in the flag manifolds. Its purpose is to see how the Lie algebra $\mathfrak{m}$ of $M$ and its center $\mathfrak{z}(\mathfrak{m}) $ can be read off from the Satake diagram of the real form~$\mathfrak{g}$. Before starting let us write some notation and facts related to the Satake diagrams. Denote by $\mathfrak{g}_{\mathbb{C}}$ the complexification of $\mathfrak{g}$ and let $\mathfrak{u}\subset \mathfrak{g}_{\mathbb{C}}$ be a~compact real form of $\mathfrak{g}_{\mathbb{C}}$ adapted to $\mathfrak{g}$, that is, $\mathfrak{g}=\mathfrak{k}\oplus \mathfrak{s}$ is a Cartan decomposition where $\mathfrak{k}=\mathfrak{g}\cap \mathfrak{u}$ and $\mathfrak{s}=\mathfrak{g}\cap {\rm i}\mathfrak{u}$. Denote by $\sigma $ and $\tau $ the conjugations in $\mathfrak{g}_{\mathbb{C}}$ with respect to $\mathfrak{u}$ and $\mathfrak{g}$ respectively. To say that $\mathfrak{u}$ is adapted to $\mathfrak{g}$ is the same as saying that these conjugations commute and the Cartan involution $\theta =\sigma \tau =\tau \sigma $ is an automorphism leaving invariant both $\mathfrak{g}$ and $\mathfrak{u}$. The subalgebra $\mathfrak{k}$ is the fixed point set of $\theta $ implying that $\tau =-1$ on ${\rm i}\mathfrak{k}$.

Starting with the maximal abelian subspace $\mathfrak{a}\subset \mathfrak{s}$ let $\mathfrak{h}\supset \mathfrak{a}$ be a Cartan subalgebra of $\mathfrak{g}$ and complexify it to the Cartan subalgebra $\mathfrak{h}_{\mathbb{C}}\subset \mathfrak{g}_{\mathbb{C}}$. The Cartan subalgebra $\mathfrak{h}$ decomposes as $\mathfrak{h}=\mathfrak{h}_{k}\oplus \mathfrak{a}$ with $\mathfrak{h}_{k}\subset \mathfrak{k}$. The Cartan subalgebras $\mathfrak{h}$ and $\mathfrak{h}_{\mathbb{C}}$ are invariant by $\sigma $, $\tau $ and $\theta =\sigma \tau $.

Denote by $\Pi _{\mathbb{C}}$ the set of roots of $ ( \mathfrak{g}_{\mathbb{C}},\mathfrak{h}_{\mathbb{C}} ) $ and by $\Pi $ the set of restricted roots of $ ( \mathfrak{g},\mathfrak{a} ) $. Each $\alpha\in \Pi $ is the restriction to $\mathfrak{a}$ of a root in $\Pi _{\mathbb{C}}$. For $\alpha \in \mathfrak{h}_{\mathbb{C}}^{\ast }$ define $H_{\alpha}\in \mathfrak{h}_{\mathbb{C}}$ by $\alpha (\cdot) =\langle H_{\alpha },\cdot \rangle $ and denote by $\mathfrak{h}_{\mathbb{R}}$ the real subspace spanned by $H_{\alpha }$, $\alpha \in \Pi _{\mathbb{C}}$. The roots are real on $\mathfrak{a}$ and purely imaginary in $\mathfrak{h}_{k}$ so that $\mathfrak{a}=\mathfrak{h}_{\mathbb{R}}\cap \mathfrak{h}$, $\mathfrak{h}_{k}={\rm i}\mathfrak{h}_{\mathbb{R}}\cap \mathfrak{h}$ and $\mathfrak{h}_{\mathbb{R}}={\rm i}\mathfrak{h}_{k}\oplus \mathfrak{a}$. The last decomposition is orthogonal with respect to the Cartan--Killing inner product in $\mathfrak{h}_{\mathbb{R}}$ because $\tau $ is an involutive isometry satisfying $\tau =-1$ in ${\rm i}\mathfrak{h}_{k}$ and $\tau =1$ in $\mathfrak{a}$.

A root $\alpha \in \Pi _{\mathbb{C}}$ is said to be imaginary if $\alpha\circ \tau =-\alpha $. Denote by $\Pi _{\func{Im}}$ the set of imaginary roots. There are the following equivalent ways to define $\Pi _{\func{Im}}$:

\begin{enumerate}\itemsep=0pt
\item A root $\alpha \in \Pi _{\mathbb{C}}$ is imaginary if and only if it annihilates on $\mathfrak{a}$. In fact, if $H\in \mathfrak{a}$ then $\tau(H) =H$ so that $\alpha (H) =\alpha ( \tau (H)) =-\alpha (H) $. Conversely if $\alpha $ is zero on $\mathfrak{a}$ then $H_{\alpha }$ is orthogonal to $\mathfrak{a}$ in $\mathfrak{h}_{\mathbb{R}}$ so that $H_{\alpha }\in {\rm i}\mathfrak{h}_{k}$ therefore
\begin{gather*}
\alpha \circ \tau (H) =\langle H_{\alpha },\tau H\rangle=\langle \tau H_{\alpha },H\rangle =-\alpha (H)
\end{gather*}
showing that $\alpha \in \Pi _{\func{Im}}$.

\item $\alpha \in \Pi _{\mathbb{C}}$ is imaginary if and only if $H_{\alpha}\in {\rm i}\mathfrak{h}_{k}$ because the decomposition $\mathfrak{h}_{\mathbb{R}}={\rm i}\mathfrak{h}_{k}\oplus \mathfrak{a}$ is orthogonal.

\item Let $H\in \mathfrak{a}$ be regular real, that is, $\beta (H ) \neq 0$ for every $\beta \in \Pi $. Then $\alpha \in \Pi _{\mathbb{C}}$ is imaginary if and only if $\alpha (H) =0$. In fact, the roots in $\Pi $ are the restrictions to $\mathfrak{a}$ of the roots in $\Pi_{\mathbb{C}}$ and $\alpha $ is imaginary if and only if it annihilates on~$\mathfrak{a}$. For this characterization the choice of the regular element $H $ is immaterial.
\end{enumerate}

To get the Satake diagram take a regular real $H\in \mathfrak{a}$ and let $\Sigma _{\mathbb{C}}\subset \Pi _{\mathbb{C}}$ be a simple system of roots such that $\alpha (H) \geq 0$ for every $\alpha \in \Sigma _{\mathbb{C}}$. Equivalently $\Sigma _{\mathbb{C}}$ is the simple system of roots associated to a Weyl chamber $\mathfrak{h}_{\mathbb{R}}^{+}$ containing $H$ in its closure. The Satake diagram is obtained from the Dynkin diagram of $\Sigma _{\mathbb{C}}$ by painting black the imaginary roots in $\Sigma _{\mathbb{C}}$ and by joining with a double arrow two roots $\alpha,\beta \in \Sigma _{\mathbb{C}}$ whose restrictions to $\mathfrak{a}$ are equal.

The set $\Sigma _{\func{Im}}$ of imaginary roots in $\Sigma _{\mathbb{C}}$ is given by
\begin{gather*}
\Sigma _{\func{Im}}=\{\alpha \in \Sigma _{\mathbb{C}}\colon \alpha (H)=0\}.
\end{gather*}If $\beta \in \Sigma _{\mathbb{C}}\setminus \Sigma _{\func{Im}}$ then $\beta (H) >0$. Hence a positive root $\gamma $ is a linear combination of $\Sigma _{\func{Im}}$ if and only if $\gamma (H) =0$ so that $\Pi _{\func{Im}}$ is the set of roots $\langle \Sigma _{\func{Im}}\rangle $ spanned by the simple imaginary roots.

The next proposition allows to reconstruct $\mathfrak{h}_{k}$ from the Satake diagram. For its statement we use the following notation

\begin{itemize}\itemsep=0pt
\item $\mathfrak{h}_{\func{Im}}$ is the subspace spanned by ${\rm i}H_{\alpha }$ with $\alpha \in \Sigma _{\func{Im}}$ (or what is the same $\alpha \in \Pi _{\func{Im}}$).

\item $\Sigma _{\mathbb{C},\mathrm{arr}}$ is the union of pairs of simple roots in a Satake diagram that are linked by a double arrow. $\Sigma _{\mathbb{C},\mathrm{arr}}^{\bot }$ is the subset of $\Sigma _{\mathbb{C},\mathrm{arr}}$ of pairs of simple roots not linked to imaginary roots.

$\Sigma _{\mathrm{arr}}$ and $\Sigma _{\mathrm{arr}}^{\bot }$ are the restrictions to $\mathfrak{a}$ of the roots in $\Sigma _{\mathbb{C},\mathrm{arr}}$ and $\Sigma _{\mathbb{C},\mathrm{arr}}^{\bot }$ respectively.

\item $\mathfrak{h}_{\mathrm{arr}}$ is the subspace spanned by ${\rm i}H_{\gamma }$ with $\gamma $ running through the set of differences $\gamma =\alpha -\beta$ with $\{\alpha,\beta\} \in \Sigma _{\mathbb{C},\mathrm{arr}}$.
\end{itemize}

\begin{Proposition}\label{propHagak}$\mathfrak{h}_{k}=\mathfrak{h}_{\func{Im}}\oplus \mathfrak{h}_{\mathrm{arr}}$.
\end{Proposition}

\begin{proof}As mentioned above there is the orthogonal direct sum $\mathfrak{h}_{\mathbb{R}}={\rm i}\mathfrak{h}_{k}\oplus \mathfrak{a}$ which implies that $\mathfrak{h}_{\func{Im}}$ $\subset \mathfrak{h}_{k}$. Also, if $\alpha $ and $\beta $ are simple roots linked by a double arrow then $\gamma =\alpha -\beta $ is zero on $\mathfrak{a}$ which means that $H_{\gamma }$ is orthogonal to $\mathfrak{a}$ so that ${\rm i}H_{\gamma }\in \mathfrak{h}_{k}$. Hence $\mathfrak{h}_{\mathrm{arr}}\subset \mathfrak{h}_{k}$. We have $\mathfrak{h}_{\func{Im}}\cap \mathfrak{h}_{\mathrm{arr}}=\{0\}$ because $\Sigma _{\mathbb{C}}$ is a basis. A dimension check shows that the sum is the whole space. In fact, $\dim \mathfrak{h}$ is the number of roots in $\Sigma _{\mathbb{C}}$ (rank of $\mathfrak{g}$) while $\dim \mathfrak{a}$ (real rank of $\mathfrak{g}$) is the number of white roots not linked plus $\dim \mathfrak{h}_{\mathrm{arr}}$ which is half the number of white linked roots. Finally $\dim \mathfrak{h}_{\func{Im}}$ is the number of black roots. Hence $\dim \mathfrak{h}_{k}=\dim \mathfrak{h}_{\func{Im}}+\dim \mathfrak{h}_{\mathrm{arr}}$ concluding the proof.
\end{proof}

The abelian subalgebra is one of the pieces of $\mathfrak{m}$. The other piece is the subalgebra generated by the imaginary roots which are described next.

\begin{Proposition}Let $\mathfrak{g}_{\func{Im}}$ be the subalgebra of $\mathfrak{g}_{\mathbb{C}}$ generated by the root spaces $( \mathfrak{g}_{\mathbb{C}})_{\alpha }$ with \mbox{$\alpha \in \Pi _{\func{Im}}$}. Then $\mathfrak{g}_{\func{Im}}$ is a complex semi-simple Lie algebra whose Dynkin diagram corresponds to the simple roots $\Sigma _{\func{Im}}$.

Put $\mathfrak{k}_{\func{Im}}=\mathfrak{g}_{\func{Im}}\cap \mathfrak{u}$. Then $\mathfrak{k}_{\func{Im}}$ is a compact real form of $\mathfrak{g}_{\func{Im}}$ and therefore it is semi-simple. Moreover, $\mathfrak{h}_{k}$ is a Cartan subalgebra of $\mathfrak{k}_{\func{Im}}$.
\end{Proposition}

\begin{proof}The first statement holds because $\Pi _{\func{Im}}$ is a~root system generated by $\Sigma _{\func{Im}}$ which is a~simple system of roots. Regarding to the subalgebra $\mathfrak{k}_{\func{Im}}$ it can be proved that $( \mathfrak{g}_{\mathbb{C}}) _{\alpha }$ is contained in $\mathfrak{k}+{\rm i}\mathfrak{k}$ if $\alpha $ is imaginary (see \cite[Lemma~14.6]{alglie}). Hence $\mathfrak{g}_{\func{Im}}\subset \mathfrak{k}+{\rm i}\mathfrak{k}$ implying that $\mathfrak{k}_{\func{Im}}\subset \mathfrak{k}$. By the Weyl construction of the compact real form applied simultaneously to $\mathfrak{g}_{\mathbb{C}} $ and~$\mathfrak{g}_{\func{Im}}$ it follows that the
intersection $\mathfrak{k}_{\func{Im}}=\mathfrak{g}_{\func{Im}}\cap \mathfrak{u}$ is a compact real form of $\mathfrak{g}_{\func{Im}}$. Finally, $\Sigma _{\func{Im}}$ is a~simple system of roots of $\mathfrak{g}_{\func{Im}}$ so that $\mathfrak{h}_{\func{Im}}$, which is spanned by ${\rm i}H_{\alpha }$ with $\alpha \in \Sigma _{\func{Im}}$, is a~Cartan subalgebra of~$\mathfrak{k}_{\func{Im}}$.
\end{proof}

Now we combine the above pieces to write down the Lie algebra $\mathfrak{m}$ from the Satake diagram of~$\mathfrak{g}$.

\begin{Proposition}\label{propSpanCenterm}$\mathfrak{m}=\mathfrak{k}_{\func{Im}}\oplus \mathfrak{z}(\mathfrak{m}) $ with $\mathfrak{k}_{\func{Im}}$ semi-simple and $\mathfrak{z}(\mathfrak{m}) $ the orthogonal complement of~$\mathfrak{h}_{\func{Im}}$ in~$\mathfrak{h}_{k}$.
\end{Proposition}

\begin{proof}The centralizer of $\mathfrak{a}$ in $\mathfrak{g}_{\mathbb{C}}$ is $\mathfrak{z}=\mathfrak{h}_{\mathbb{C}}\oplus \sum\limits_{\alpha \in \Pi _{\func{Im}}} ( \mathfrak{g}_{\mathbb{C}} ) _{\alpha }$. Thus $\mathfrak{m}=\mathfrak{z}\cap \mathfrak{k}=\mathfrak{h}_{k}\oplus \mathfrak{k}_{\func{Im}} $. Since $\mathfrak{h}_{\func{Im}}\subset \mathfrak{k}_{\func{Im}}$ we have $\mathfrak{m}=\mathfrak{h}_{\func{Im}}^{\bot }\oplus \mathfrak{k}_{\func{Im}}$ where $\mathfrak{h}_{\func{Im}}^{\bot }$ is the orthogonal complement of $\mathfrak{h}_{\func{Im}}$ in $\mathfrak{h}_{k}$. Any imaginary root is zero on $\mathfrak{h}_{\func{Im}}^{\bot }$ so that this subspace commutes with $\mathfrak{k}_{\func{Im}}$. This implies that $\mathfrak{z}(\mathfrak{m}) =\mathfrak{h}_{\func{Im}}^{\bot }$ because $\mathfrak{k}_{\func{Im}}$ is semi-simple.
\end{proof}

By Proposition~\ref{propHagak} we have $\mathfrak{h}_{k}=\mathfrak{h}_{\func{Im}}\oplus \mathfrak{h}_{\mathrm{arr}}$ so that $\mathfrak{h}_{\func{Im}}^{\bot }\neq \{0\}$ if and only if $\mathfrak{h}_{\mathrm{arr}}\neq \{0\}$
which means that the Satake diagram of $\mathfrak{g}$ has double arrows. Diagrams with arrows are called outer diagrams (because $\theta $ is an outer automorphism of $\mathfrak{g}_{\mathbb{C}}$). By the above proposition
$\mathfrak{z}(\mathfrak{m}) =\mathfrak{h}_{\func{Im}}^{\bot }$ so we get the following case where $\mathfrak{z}(\mathfrak{m}) $ is not trivial.

\begin{Proposition}\label{propCentrom}$\mathfrak{m}$ has non-trivial center if and only if the Satake diagram of $\mathfrak{g}$ is outer and $\mathfrak{z} ( \mathfrak{m}) =\mathfrak{h}_{\func{Im}}^{\bot }$.
\end{Proposition}

\begin{Remark}For $\alpha,\beta\in \Sigma_{\mathbb{C},\mathrm{arr}}^\bot$ and $\gamma=\alpha-\beta$, ${\rm i}H_\gamma\in \mathfrak{h}_{\func{Im}}^{\bot }=\mathfrak{z }(\mathfrak{m })$. In fact, if $\delta$ is an imaginary root,
then $\langle {\rm i}H_\gamma,{\rm i}H_\delta \rangle=0$ since $\alpha$ and $\beta$ are orthogonal (non-adjacent) to any imaginary root.
\end{Remark}

The diagrams $\mathrm{AIII}_{1}$, $\mathrm{AIII}_{2}$, $\mathrm{DI}_{2}$ and $\mathrm{EII}$ appearing in Theorem~\ref{tedeRham} whose Lie algebras have flag manifolds with non-trivial $2$-cohomology are outer diagrams. We
reproduce them below.

\centerline{\includegraphics{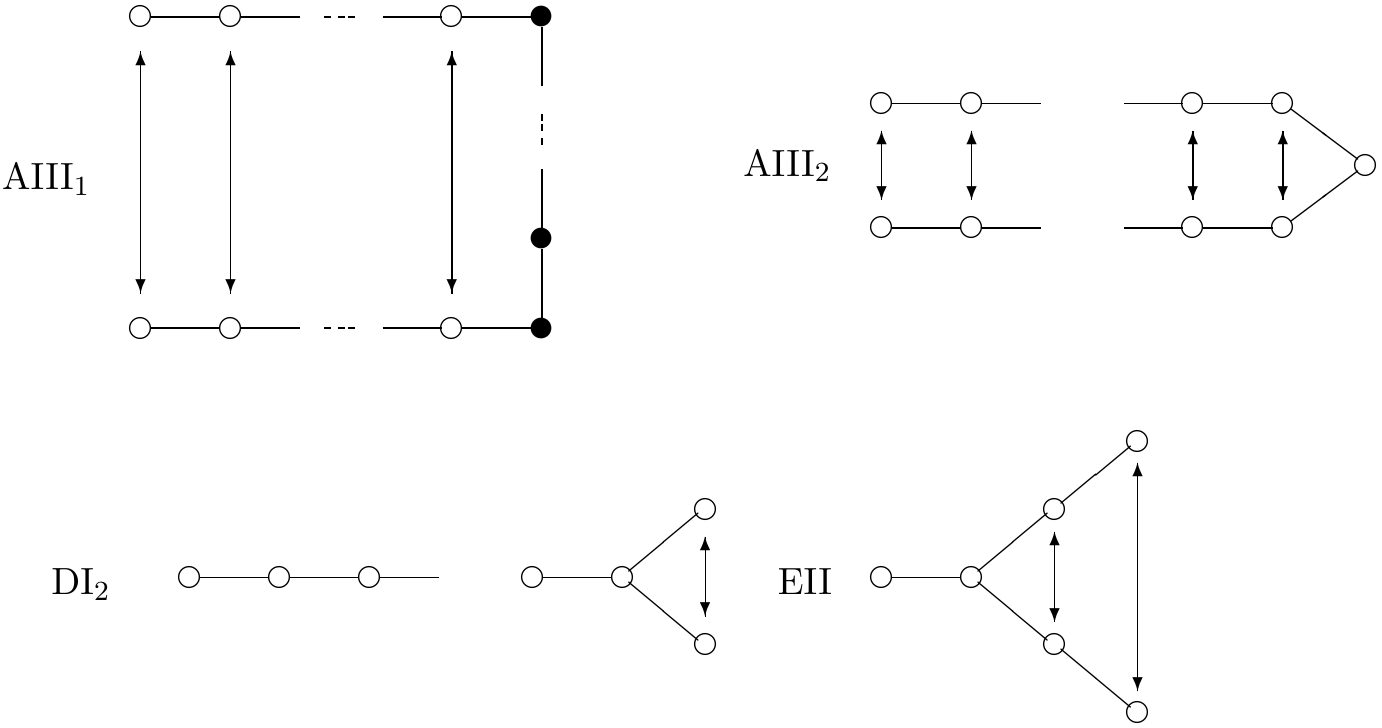}}

In these diagrams a simple restricted root $\alpha \in \Sigma $ has $\operatorname{rank} \alpha =2$ if and only if it is the restriction of $\gamma,\delta \in \Sigma _{\mathbb{C},\mathrm{arr}}^{\bot }$, that is, we have $\Sigma_{2}=\Sigma _{\mathrm{arr}}^{\bot }$.

The diagrams $\mathrm{AIII}_{2}$, $\mathrm{DI}_{2}$ and $\mathrm{EII}$ have no imaginary roots. Hence in these Lie algebras $\mathfrak{m=h}_{k}$ is abelian. For $\mathrm{AIII}_{1}$ the imaginary simple roots form an $A_{l}$
Dynkin diagram. So that in $\mathrm{AIII}_{1}$ we have $\mathfrak{k}_{\func{Im}}\approx \mathfrak{su}(k) $ some $k$, that is, $\mathfrak{m}=\mathfrak{su}(k) \oplus \mathfrak{z}(\mathfrak{m}) $ where $\dim \mathfrak{z}(\mathfrak{m}) $ is the number of pairs of roots linked double arrows and equals the real rank of the Lie algebra.

Now we use the well known fact that every connected component of $M$ contains an element that commutes with $\mathfrak{z} ( \mathfrak{m} ) $ to conclude that $\mathfrak{z}(\mathfrak{m}) $ is fixed by the whole group $M$.

\begin{Proposition}Let $\mathbb{F}=G/P=K/M$ be the maximal flag manifold associated to the real form $\mathfrak{g}$ of the complex simple Lie algebra $\mathfrak{g}_{\mathbb{C}}$. Then the center $\mathfrak{z}(\mathfrak{m}) $ of the Lie algebra $\mathfrak{m}$ of $M$ is non-trivial if $\mathfrak{g}$ is of the type $\mathrm{AIII}_{1}$, $\mathrm{AIII}_{2}$, $\mathrm{DI}_{2}$ or $\mathrm{EII}$. In these cases $M$ centralizes $\mathfrak{z}(\mathfrak{m}) $.
\end{Proposition}

\begin{proof}It remains to check only the last statement. For the realization $\mathbb{F}=G/P$ we can take any connected Lie group $G$ with Lie algebra $\mathfrak{g}$. If $G$ is complexifiable then any connected component of $M$ has an element of the form ${\rm e}^{{\rm i}H}$ with $H\in \mathfrak{a}$ (see Knapp \cite[Section~VII.5]{knapp} for the precise statement of this result and the notion of complexifiable group). If $X\in \mathfrak{z}(\mathfrak{m}) $ then $\operatorname{Ad} \big({\rm e}^{{\rm i}H} \big) X=X$ because $\mathfrak{z} (\mathfrak{m} ) $ is contained $\mathfrak{h}_{k}$ that commutes with $\mathfrak{a}$. So that $M$ fixes $\mathfrak{z}(\mathfrak{m}) $. In general, we can take the adjoint representation and find elements in the several connected components of $M$ that have the form ${\rm e}^{{\rm i}H}z$ with $z$ in the center of $G$. Hence the result follows as well.
\end{proof}

To conclude this section we describe subalgebras of the real forms that are associated to rank~$2$ roots and are isomorphic to $\mathfrak{sl}(2,\mathbb{C} ) $. The following lemma can be checked by looking at the
table of Satake diagrams.

\begin{Lemma}Let $\alpha $ and $\beta $ be simple roots in a Satake diagram that are linked by a double arrow. Then $\langle \alpha,\beta \rangle =0$.
\end{Lemma}

If $\mathbb{\alpha }$ is a root of a complex simple Lie algebra $\mathfrak{g}_{\mathbb{C}}$ then the subalgebra $\mathfrak{g}_{\mathbb{C}} (\alpha) $ generated by the root spaces $( \mathfrak{g}_{\mathbb{C}}) _{\pm \alpha }$ is isomorphic to $\mathfrak{sl}( 2,\mathbb{C}) $. We prove next that the same happens to certain rank~$2$ roots of a real form.

\begin{Proposition}\label{propEsseeleC}Let $\gamma $ and $\delta $ be simple roots in the Satake diagram of real form $\mathfrak{g}$. Suppose that~$\gamma $ and~$\delta $ are linked by a double arrow and not linked to imaginary roots, that is, $\gamma,\delta \in \Sigma _{\mathbb{C},\mathrm{arr}}^{\bot }$. If $\alpha $ denotes their common restrictions to $\mathfrak{a}$ then $\operatorname{rank} \alpha =2$. Let $\mathfrak{g} (\alpha) $ be the subalgebra of~$\mathfrak{g}$ generated by the root spaces $\mathfrak{g}_{\pm \alpha }$. Then $\mathfrak{g}(\alpha) $ is isomorphic to the realification $\mathfrak{sl}(2,\mathbb{C}) _{R}$ of~$\mathfrak{sl}(2,\mathbb{C}) $.

The subspace $\mathfrak{h}_{\alpha }$ spanned over $\mathbb{R}$ by $H_{\alpha }\in \mathfrak{a}$ and ${\rm i}H_{\gamma -\delta }\in \mathfrak{h}_{k}$ is a Cartan subalgebra of $\mathfrak{g}(\alpha) $.
\end{Proposition}

\begin{proof}We have $\mathfrak{g}_{\alpha }=( ( \mathfrak{g}_{\mathbb{C}}) _{\gamma }+( \mathfrak{g}_{\mathbb{C}}) _{\delta}) \cap \mathfrak{g}$ and the same for $-\alpha $. By assumption $\gamma $ and $\delta $ are the only roots restricting to $\alpha $ hence $\mathfrak{g}(\alpha) =\mathfrak{g}_{\mathbb{C}}( \gamma,\delta ) \cap \mathfrak{g}$ where $\mathfrak{g}_{\mathbb{C}}(\gamma,\delta) $ is the Lie algebra generated by $( \mathfrak{g}_{\mathbb{C}}) _{\pm \gamma }$ and $( \mathfrak{g}_{\mathbb{C}}) _{\pm \delta }$. By the previous lemma $\langle \gamma,\delta\rangle =0$ which implies that $\mathfrak{g}_{\mathbb{C}}( \gamma,\delta ) =\mathfrak{g}_{\mathbb{C}}(\gamma) \oplus\mathfrak{g}_{\mathbb{C}}(\delta) $ with $[\mathfrak{g}_{\mathbb{C}}(\gamma),\mathfrak{g}_{\mathbb{C}}( \delta) ]=0$ so that $\mathfrak{g}_{\mathbb{C}}( \gamma,\delta)
$ is isomorphic to the direct sum $\mathfrak{sl}( 2,\mathbb{C}) \oplus \mathfrak{sl}( 2,\mathbb{C}) $. Hence $\mathfrak{g}(\alpha) =( \mathfrak{g}_{\mathbb{C}}(\gamma)\oplus \mathfrak{g}_{\mathbb{C}}(\delta) ) \cap \mathfrak{g}$ and $\mathfrak{g}_{\mathbb{C}}(\gamma) \oplus \mathfrak{g}_{\mathbb{C}}(\delta) $ is the complexification of $\mathfrak{g}(\alpha) $. It follows that $\mathfrak{g}(\alpha) $ is isomorphic to $\mathfrak{sl}( 2,\mathbb{C}) _{R}$ because this is the only non-compact real form of $\mathfrak{sl} ( 2,\mathbb{C} ) \oplus \mathfrak{sl}(2,\mathbb{C}) $.

The subspace $\mathfrak{h}_{\gamma,\delta }=\operatorname{span}_{\mathbb{C}}\{H_{\gamma },H_{\delta }\}$ is a Cartan subalgebra of $\mathfrak{g}_{\mathbb{C}}(\gamma) \oplus \mathfrak{g}_{\mathbb{C}}(\delta) $. Hence
\begin{gather*}
\mathfrak{h}_{\alpha }=\mathfrak{h}_{\gamma,\delta }\cap \mathfrak{g}=\mathfrak{h}_{\gamma,\delta }\cap ( \mathfrak{h}_{k}\oplus \mathfrak{a})
\end{gather*}
is a Cartan subalgebra of $\mathfrak{g}(\alpha) $.
\end{proof}

\section{Characteristic forms}\label{secCharacClass}

In this section we apply the Weil homomorphism to get representatives of the de Rham cohomo\-lo\-gy of the flag manifolds of the Lie algebras $\mathrm{AIII}_{1}$, $\mathrm{AIII}_{2}$, $\mathrm{DI}_{2}$ and $\mathrm{EII}$ appearing in Theorem~\ref{tedeRham}.

Connections in the principal bundles $K\rightarrow \mathbb{F}_{\Theta}=K/K_{\Theta }$ will be obtained from the following general standard construction.

\begin{Proposition}Let $Q$ be a Lie group and $L\subset Q$ a closed subgroup with Lie algebras $\mathfrak{q}$ and $\mathfrak{l}$ respectively. Suppose there exists a subspace $\mathfrak{p}\subset \mathfrak{q}$ with $\mathfrak{q}=\mathfrak{l}\oplus \mathfrak{p}$ and $\operatorname{Ad}(g) \mathfrak{p}=\mathfrak{p}$ for all $g\in L$. Then
\begin{enumerate}\itemsep=0pt
\item[$1.$] The left translations $\operatorname{Hor}(q) =L_{q\ast } ( \mathfrak{p}) $, $q\in Q$, are horizontal spaces for a connection in the principal bundle $\pi\colon Q\rightarrow Q/L$.

\item[$2.$] The connection form $\omega $ is given by $\omega _{q}\big(X^{l}(q) \big) =PX\in \mathfrak{l}$ where $X^{l}$ is the left invariant vector field defined by $X\in \mathfrak{q}$ and $P\colon \mathfrak{q}\rightarrow \mathfrak{l}$ is the projection against the decomposition $\mathfrak{q}=\mathfrak{l}\oplus \mathfrak{p}$.

\item[$3.$] The curvature form $\Omega $ is completely determined by its restriction to $\mathfrak{p}$ $($at the origin$)$ which is given by $\Omega \big( X^{l},Y^{l}\big) =P [ X,Y ] $.
\end{enumerate}

The horizontal distribution $\operatorname{Hor}$ as well as $\omega $ and $\Omega $ are invariant by left translations $(L_{g\ast }\operatorname{Hor} \allowbreak=\operatorname{Hor}$, $L_{g}^{\ast }\omega =\omega $ and $L_{g}^{\ast }\Omega =\Omega )$.
\end{Proposition}

Left invariant connections on the bundles $K\rightarrow \mathbb{F}_{\Theta }=K/K_{\Theta }$ given by this proposition will be used to get characteristic forms on the flag manifolds $\mathbb{F}_{\Theta }$.

We work out first the maximal flag manifolds $\mathbb{F}=K/M$. Take the root space decomposition
\begin{gather*}
\mathfrak{g}=\mathfrak{m}\oplus \mathfrak{a}\oplus \sum_{\alpha \in \Pi }\mathfrak{g}_{\alpha }
\end{gather*}
and for a positive restricted root $\alpha \in \Pi ^{+}$ write $\mathfrak{k}_{\alpha }= ( D_{\alpha }+D_{-\alpha } ) \cap \mathfrak{k}$ where $D_{\alpha }=\mathfrak{g}_{\alpha /2}+\mathfrak{g}_{\alpha }+\mathfrak{g}_{2\alpha }$. We have $\mathfrak{k}=\mathfrak{m}\oplus \sum\limits_{\alpha \in \Pi ^{+}}\mathfrak{k}_{\alpha }$ so that
\begin{gather*}
\mathfrak{p}=\sum_{\alpha \in \Pi ^{+}}\mathfrak{k}_{\alpha }
\end{gather*}
complements $\mathfrak{m}$ in $\mathfrak{k}$. For any $m\in M$ there is the invariance $\operatorname{Ad}(m) \mathfrak{p}=\mathfrak{p}$ because $\operatorname{Ad}(m) \mathfrak{g}_{\alpha }=\mathfrak{g}_{\alpha }$ for every root $\alpha $. Hence $\mathfrak{p}$ defines a left invariant connection in the bundle $K\rightarrow K/M$. Its curvature form $\Omega $ is defined at the identity by the projection into $\mathfrak{m}$ of the bracket. Since $[\mathfrak{g}_{\alpha },\mathfrak{g}_{\beta }]\subset \mathfrak{g}_{\alpha +\beta }$ it follows that $\Omega ( X,Y) =0$ if $X\in \mathfrak{g}_{\alpha }$, $Y\in \mathfrak{g}_{\beta }$ with $\alpha \neq \beta $. Hence $\Omega $ (at the origin) is the direct sum of $2$-forms $\Omega ^{\alpha }$ on the spaces $\mathfrak{k}_{\alpha }$:
\begin{gather*}
\Omega =\sum_{\alpha \in \Pi ^{+}}\Omega ^{\alpha },
\end{gather*}%
where $\Omega ^{\alpha }$ is the projection into $\mathfrak{m}$ of the bracket in $\mathfrak{k}_{\alpha }$.

If $Z\in \mathfrak{z}(\mathfrak{m}) $ then by the Weil homomorphism theorem the left invariant $2$-form $\langle Z,\Omega (\cdot,\cdot ) \rangle $ in~$K$ is the pull back of a closed differential form in $\mathbb{F}=K/M$ that we denote by $f_{Z}$. The tangent space at the origin identifies with $\mathfrak{p}$ where $f_{Z}$ is given by
\begin{gather*}
f_{Z}(X,Y) =\sum_{\alpha \in \Pi ^{+}}\langle Z,\Omega ^{\alpha}(X,Y) \rangle.
\end{gather*}%
We write $f_{Z} ( S_{r_{\alpha }} ) $ for the value of $f_{Z}$ in the Schubert cell $S_{r_{\alpha }}$ given by the duality between $H^{2}( \mathbb{F},\mathbb{R}) $ and $H_{2}( \mathbb{F},\mathbb{R}) $.

The following example of $\mathbb{C}P^{1}\approx S^{2}$ will be used to compute the values of characteristic forms in the Schubert $2$-cells.

\begin{Example} Let us look at the characteristic forms at the complex projective line $\mathbb{C}P^{1}\approx S^{2}$ obtained by the action of $\mathrm{SU}(2) $ so that $S^{2}=\mathrm{SU}(2) /T$ where $T$ is the group of diagonal matrices $\operatorname{diag}\big\{{\rm e}^{{\rm i}t},{\rm e}^{-{\rm i}t}\big\}$, $t\in \mathbb{R}$. The Lie algebra of $T$ is $\mathfrak{t}=\{H_{t}=\operatorname{diag}\{{\rm i}t,-{\rm i}t\}\colon t\in \mathbb{R}\}$ and the subspace
\begin{gather*}
\mathfrak{p}=\left\{X_{z}=\left(
\begin{matrix}
0 & z \\
-\overline{z} & 0
\end{matrix}%
\right) \colon z\in \mathbb{C}\right\}
\end{gather*}%
complements $\mathfrak{t}$ in $\mathfrak{su}(2) $. The bracket in $\mathfrak{p}$ is given by
\begin{gather*}
[ X_{z},X_{w}] =\left(
\begin{matrix}
-z\overline{w}+\overline{z}w & 0 \\
0 & z\overline{w}-\overline{z}w
\end{matrix}%
\right) =2\left(
\begin{matrix}
-{\rm i}\func{Im}z\overline{w} & 0 \\
0 & {\rm i}\func{Im}z\overline{w}
\end{matrix}
\right).
\end{gather*}
Hence by taking the appropriate inner product in $\mathfrak{t}$ we have that $f_{H_{1}}$ is the only invariant $2$-form in~$S^{2}$ that in the origin satisfies $f_{H_{1}} ( X_{1},X_{i} ) =1$. Thus $f_{H_{1}}$ is the unique (up to scale) $\mathrm{SU}(2) $-invariant volume form in~$S^{2}$. We can normalize the inner product in~$\mathfrak{t}$ so that the integral of~$f_{H_{1}}$ over $S^{2}$ equals to~$1$.
\end{Example}

Now we can compute the basis of $H^{2}( \mathbb{F},\mathbb{R}) $ dual to the Schubert cells spanning $H_{2} ( \mathbb{F},\mathbb{R} ) $ when $\mathbb{F}$ is the maximal flag manifold of one of the Lie algebras $\mathrm{AIII}_{1}$, $\mathrm{AIII}_{2}$, $\mathrm{DI}_{2}$ and $\mathrm{EII}$. Recall that by Proposition~\ref{propCentrom} for these Lie algebras $\mathfrak{z}(\mathfrak{m}) =\mathfrak{h}_{\func{Im}}^{\bot }\neq \{0\}$. On the other hand, $H_{2} ( \mathbb{F},\mathbb{R} ) $ is spanned by the $2$-cells $S_{r_{\alpha }}\approx S^{2}$ with $\alpha $ running through $\Sigma _{\mathrm{arr}}^{\bot }$ as in Proposition~\ref{propEsseeleC}.

To get a basis of $H^{2}(\mathbb{F},\mathbb{R}) $ dual to the basis of homology given by the Schubert cells we define the Weil map $\mathfrak{z}(\mathfrak{m}) \rightarrow H^{2}( \mathbb{F},\mathbb{R}) $ by means of a suitable normalization $ ( \cdot,\cdot ) _{c}=c\langle \cdot,\cdot \rangle $ of an invariant inner product in $\mathfrak{k}$. The normalizing constant $c$ is chosen so that $f_{{\rm i}H_{\gamma -\delta }} ( S_{r_{\alpha }} ) =1$ for every simple root $\alpha \in \Sigma _{\mathrm{arr}}^{\bot }$ where $\gamma,\delta \in \Sigma _{\mathbb{C},\mathrm{arr}}^{\bot }$ are such that $\alpha =\gamma _{\vert \mathfrak{a} }=\delta _{\vert \mathfrak{a}}$. The choice of $c$ is possible because the Satake diagrams $\mathrm{AIII}_{1}$, $\mathrm{AIII}_{2}$, $\mathrm{DI}_{2}$ and $\mathrm{EII}$ have only simple links and hence the roots are all of the same length.

To state the next theorem we take the basis $\mathcal{B}$ of $\mathfrak{h}_{k}$ given by
\begin{gather*}
\mathcal{B}=\big\{{\rm i}H_{\gamma _{1}-\delta _{1}}^{N},\ldots,{\rm i}H_{\gamma _{k}-\delta _{k}}^{N}\big\}\cup \big\{{\rm i}H_{\gamma -\delta }^{N}\big\}\cup \big\{{\rm i}H_{\mu_{1}}^{N},\ldots,{\rm i}H_{\mu _{s}}^{N}\big\},
\end{gather*}
where for $Z\in \mathfrak{h}_{k}$, $Z^{N}=Z/(Z,Z) _{c}$ and we are using the following notation:
\begin{enumerate}\itemsep=0pt
\item $\{\mu _{1},\ldots,\mu _{s}\}$ are the imaginary simple roots so that $\{{\rm i}H_{\mu _{1}},\ldots,{\rm i}H_{\mu _{s}}\}$ is a basis of $\mathfrak{h}_{\func{Im}}$.

\item $\gamma _{j}$, $\delta _{j}$, $j=1,\ldots,k$ are the pairs of roots in $\Sigma _{\mathbb{C},\mathrm{arr}}^{\bot }$, that is, $\gamma _{j}$ is linked to $\delta _{j}$ by a double arrow and both are not linked to imaginary
roots.

\item $\gamma$, $\delta $ is the only pair of roots in $\Sigma _{\mathbb{C},\mathrm{arr}}$ that are linked to imaginary roots (which occurs only in $\mathrm{AIII}_{1}$).
\end{enumerate}

Now let $\mathcal{B}^{\bot }$ be the dual basis of $\mathcal{B}$ (with
respect to the normalized form $(\cdot,\cdot) _{c}$) and
denote by $\{Z_{1},\ldots,Z_{k}\}$ the first $k$ elements of $\mathcal{B}%
^{\bot }$ that correspond to the roots in $\Sigma _{\mathrm{arr}}^{\bot }$.

\begin{Theorem}\label{teo.baseZ} Let $\{Z_{1},\ldots,Z_{k}\}$ be as above. Then $\{f_{Z_{1}},\ldots, f_{Z_{k}}\}$ is the basis of $H^{2}( \mathbb{F},\mathbb{R}) $ dual to the basis $\{S_{r_{\alpha }}\colon \alpha \in \Sigma _{\mathrm{arr}}^{\bot }\}$ of $H_{2}(\mathbb{F},\mathbb{R}) $.
\end{Theorem}

\begin{proof}Take $\alpha \in \Sigma _{\mathrm{arr}}^{\bot }$ and let $\gamma,\delta \in\Sigma _{\mathbb{C},\mathrm{arr}}^{\bot }$ be such that $\alpha =\gamma_{\vert \mathfrak{a} }=\delta _{\vert \mathfrak{a}}$. By Proposition \ref{propEsseeleC} we have $\mathfrak{g}( \alpha) \approx \mathfrak{sl}(2,\mathbb{C}) $. Put $G(\alpha) =\langle \exp \mathfrak{g}(\alpha) \rangle $, $\mathfrak{k}(\alpha) =\mathfrak{g}(\alpha) \cap
\mathfrak{k}$ and $K(\alpha) =\langle \exp \mathfrak{k}(\alpha) \rangle =G(\alpha) \cap K$. We have $S_{r_{\alpha }}=G(\alpha) x_{0}=K(\alpha) x_{0}$ where $x_{0}$ is the origin of the flag manifold $\mathbb{F}$ (see \cite[Proposition 1.3]{RSM}). Let $\phi _{\alpha }\colon \mathfrak{su}(2) \rightarrow \mathfrak{k}(\alpha) $ be an isomorphism assured by Proposition~\ref{propEsseeleC} and put
\begin{gather*}
X_{\alpha }=\phi _{\alpha }\left(
\begin{matrix}
0 & 1 \\
-1 & 0%
\end{matrix}
\right), \qquad Y_{\alpha }=\phi _{\alpha }\left(
\begin{matrix}
0 & {\rm i} \\
{\rm i} & 0
\end{matrix}
\right).
\end{gather*}
We have
\begin{gather*}
\lbrack X_{\alpha },Y_{\alpha }]=2\phi _{\alpha }\left(
\begin{matrix}
{\rm i} & 0 \\
0 & -{\rm i}
\end{matrix}
\right) =2{\rm i}H_{\gamma -\delta }
\end{gather*}
so that $\Omega ( X_{\alpha },Y_{\alpha }) =2{\rm i}H_{\gamma -\delta }$. Now take ${\rm i}H\in \operatorname{span}_{\mathbb{R}}\mathcal{B}$. Then
\begin{gather*}
f_{{\rm i}H}( X_{\alpha },Y_{\alpha }) =( {\rm i}H,\Omega (X_{\alpha },Y_{\alpha })) _{c}=( {\rm i}H,2{\rm i}H_{\gamma -\delta}) _{c}.
\end{gather*}%
In particular $f_{{\rm i}H_{\gamma -\delta }} ( X_{\alpha },Y_{\alpha } ) =2 ( {\rm i}H_{\gamma -\delta },{\rm i}H_{\gamma -\delta } ) _{c}$ so that
\begin{gather*}
f_{{\rm i}H}( X_{\alpha },Y_{\alpha }) =\frac{( {\rm i}H,{\rm i}H_{\gamma-\delta }) _{c}}{( {\rm i}H_{\gamma -\delta },{\rm i}H_{\gamma -\delta}) _{c}}f_{{\rm i}H_{\gamma -\delta }}( X_{\alpha },Y_{\alpha })
=\big( {\rm i}H,{\rm i}H_{\gamma -\delta }^{N}\big) _{c}f_{{\rm i}H_{\gamma -\delta}}( X_{\alpha },Y_{\alpha }).
\end{gather*}
Hence the restriction to $S_{r_{\alpha }}$ yields%
\begin{gather*}
f_{{\rm i}H} ( S_{r_{\alpha }} ) =\big( {\rm i}H,{\rm i}H_{\gamma -\delta }^{N}\big) _{c}f_{{\rm i}H_{\gamma -\delta }} ( S_{r_{\alpha }} ) =\big( {\rm i}H,{\rm i}H_{\gamma -\delta }^{N}\big) _{c}.
\end{gather*}
Since $\{Z_{1},\ldots,Z_{k}\}$ is the basis dual to the basis $\big\{{\rm i}H_{\gamma_{1}-\delta _{1}}^{N},\ldots,{\rm i}H_{\gamma _{k}-\delta _{k}}^{N}\big\}$ with respect to $(\cdot,\cdot) _{c}$ it follows that $f_{Z_{j}} ( S_{r_{\alpha _{k}}} ) =1$ if $j=k$ and $0$ otherwise so that $\{f_{Z_{1}},\ldots, f_{Z_{k}}\}$ is the basis dual to $\{S_{r_{\alpha}}\colon \alpha \in \Sigma _{\mathrm{arr}}^{\bot }\}$, concluding the proof.
\end{proof}

This result holds also for a partial flag manifold $\mathbb{F}_{\Theta}=K/K_{\Theta }$ by taking roots in $\Sigma _{2}=\Sigma _{\mathrm{arr}}^{\bot }$ that are outside $\Theta $. By Theorem~\ref{teoResumo} the $2$-homology $H_{2}( \mathbb{F}_{\Theta },\mathbb{R}) $ is generated by the Schubert cells $S_{r_{\alpha }}^{\Theta }$ with $\alpha \in \Sigma _{\mathrm{arr}}^{\bot }\setminus \Theta $. On the other hand, let $\{Z_{1},\ldots,Z_{k}\}$ be the dual basis as in the statement of the above theorem and take an index $j$ such that both roots $\gamma _{j}$ and $\delta _{j}$ restrict to $\alpha _{j}\in \Sigma _{\mathrm{arr}}^{\bot }\setminus \Theta $. To prove that forms $f_{Z_{j}}$ corresponding to these indices form a dual basis in~$H^{2} ( \mathbb{F}_{\Theta },\mathbb{R} ) $ it remains to check that $Z_{j}$ belong to the center $\mathfrak{z}(\mathfrak{k}_{\Theta }) $ so that the forms $f_{Z_{j}}$ are well defined in $\mathbb{F}_{\Theta }$.

\begin{Lemma} Let $Z_{j}$, $j=1,\ldots,k$ be as in the above theorem the elements of the dual basis and take an index $j$ that corresponds to $\alpha _{j}\in \Sigma _{\mathrm{arr}}^{\bot }\setminus \Theta $. Then $Z_{j}\in \mathfrak{z}(\mathfrak{k}_{\Theta }) $.
\end{Lemma}

\begin{proof}By Proposition \ref{propSpanCenterm} we have $Z_{j}\in \mathfrak{z} ( \mathfrak{m}) \subset \mathfrak{m}\subset \mathfrak{k}_{\Theta }$. Denote by $\Theta _{\mathbb{C}}\subset \Sigma _{\mathbb{C}}$ the set of
roots of the Satake diagram whose restrictions to $\mathfrak{a}$ belong to $\Theta \cup \{0\}$. Take $\gamma \in \Theta _{\mathbb{C}}$ and let $\alpha \in \Theta \cup \{0\}$ be its restriction. We claim that $\gamma (Z_{j}) =0$. There are the possibilities:
\begin{enumerate}\itemsep=0pt
\item $\alpha =0$, that is, $\gamma $ is imaginary. Then $\gamma ( Z_{j}) =0$ because $Z_{j}$ is orthogonal to ${\rm i}H_{\gamma }$.

\item $\alpha $ has multiplicity $1$ so that $\gamma =\alpha $ is not imaginary and not linked to another root of $\Sigma _{\mathbb{C}}$ by a~double arrow. Then $H_{\gamma }\in \mathfrak{a}$ and since ${\rm i}Z_{j}\in
\mathfrak{a}^{\bot }$ we have $\gamma ( Z_{j}) =0$.

\item $\gamma $ is linked to $\delta $ by a double arrow. By definition of the dual basis $\mathcal{B}^{\bot}$ we have that~$Z_{j}$ is orthogonal to ${\rm i}H_{\gamma -\delta }$. On the other hand, $H_{\gamma +\delta }\in \mathfrak{a}$ since $\alpha = ( \gamma +\delta ) /2$. Hence $\langle H_{\gamma +\delta },Z_{j}\rangle =0$ so that $\gamma ( Z_{j}) =0$ as claimed.
\end{enumerate}

It follows that $Z_{j}$ centralizes the Lie algebra $\mathfrak{g}_{\mathbb{C}} ( \Theta _{\mathbb{C}} ) $ generated by $ ( \mathfrak{g}_{\mathbb{C}} ) _{\gamma }$, $\gamma \in \Theta _{\mathbb{C}}$. Hence $Z_{j}$ centralizes $\mathfrak{k}_{\Theta }=\mathfrak{g}_{\mathbb{C}} (\Theta _{\mathbb{C}}) \cap \mathfrak{k}$ as well, concluding the proof.
\end{proof}

The $2$-forms $\{f_{Z_{1}},\ldots,f_{Z_{k}}\}$ are zero on any $\mathfrak{k}(\alpha )$ with $\alpha $ not of rank $2$. Therefore, if there is a closed form on $H^{2}(\mathbb{F}_{\Theta },\mathbb{R})$ which is non-degenerate on $\mathbb{F}_{\Theta }$ then $\Pi \backslash \Theta \subset \Sigma _{\mathrm{arr}}^{\bot }$. This implies that $\Pi \backslash \Theta $ consists of roots with double arrows in the Satake diagram. In this case, $\mathbb{F}_{\Theta } $ is a product of complex flag manifolds of the form ${\rm SU}(n)/T$.

\section[$\mathfrak{su}(p,q)$]{$\boldsymbol{\mathfrak{su}(p,q)}$} \label{secEsseu}

In this section we present concrete realizations of the flag manifolds of the Lie algebras of types $\mathrm{AIII}_{1}$ and $\mathrm{AIII}_{2}$. These correspond to $\mathfrak{su}(p,q)$ with $p\leq q$ which are non-compact real forms of $\mathfrak{sl}(p+q,\C)$.

The Lie algebras $\mathfrak{su}(p,q) $, $p\leq q$ are constituted by zero trace matrices which are skew-hermitian with respect to the hermitian form in $\mathbb{C}^{p+q}$ with matrix
\begin{gather*}
J_{p,q}=\left(
\begin{matrix}
0 & 1_{p\times p} & 0 \\
1_{p\times p} & 0 & 0 \\
0 & 0 & 1_{q-p}\times 1_{q-p}%
\end{matrix}
\right).
\end{gather*}
That is, $\mathfrak{su}(p,q) =\{Z\in M_{p+q} ( \mathbb{C} ) \colon ZJ_{p,q}+J_{p,q}Z^{\ast }=0,\, \operatorname{tr}Z=0\}$ where $Z^{\ast }$ denotes the transpose conjugate. Therefore, $\mathfrak{su}(p,q) $ is the Lie algebra of $(p+q) \times (p+q) $ matrices of the following form
\begin{gather}
\left(
\begin{matrix}
A & B & -Y^{\ast } \\
C & -A^{\ast } & -X^{\ast } \\
X & Y & Z%
\end{matrix}
\right), \qquad
\begin{array}{@{}l}
A\in \mathfrak{gl}(p,\mathbb{C}),\qquad B,C\in \mathfrak{u}(p), \\
Z\in \mathfrak{u}(q-p), \qquad \func{tr}(2\func{Im}A+Z)=0.
\end{array}\label{formatrizes}
\end{gather}
The Lie algebra $\mathfrak{su}(p,q) $ can also be realized as the set of skew-hermitian matrices with respect to the bilinear form with matrix
\begin{gather*}
I_{p,q}=\left(
\begin{matrix}
1_{p\times p} & 0 \\
0 & -1_{q\times q}%
\end{matrix}
\right).
\end{gather*}
By using either realization one can see that the complexified Lie algebra $\mathfrak{su}(p,q) _{\mathbb{C}}$ is $\mathfrak{sl} ( p+q,\mathbb{C} ) $. From the second realization it is clear that in the Cartan decomposition $\mathfrak{g}=\mathfrak{k}\oplus \mathfrak{s}$, $\mathfrak{k}$ is isomorphic to $( \mathfrak{u}(p) \oplus \mathfrak{u}(q)) /\operatorname{tr}$, that is, zero trace matrices given by diagonal block matrices with two diagonal elements in $\mathfrak{u}(p) $ e $\mathfrak{u}(q) $, respectively. In the first realization~(\ref{formatrizes}), $\mathfrak{k}$ is described as follows
\begin{gather*}
\mathfrak{k}=\left\{\left(
\begin{matrix}
A & B & -X^{\ast } \\
B & A & -Y^{\ast } \\
X & Y & Z%
\end{matrix}
\right) \colon A,B\in \mathfrak{u}(p), \ Z\in \mathfrak{u} ( q-p), \ \func{tr}(2\func{Im}A+Z)=0\right\}.
\end{gather*}%
The matrices in $\mathfrak{s}$ are hermitian, that is,
\begin{gather*}
\mathfrak{s}=\left\{\left(
\begin{matrix}
A & -B & 0 \\
B & -A & 0 \\
0 & 0 & 0%
\end{matrix}
\right) \colon A=A^{\ast }, \ B\in \mathfrak{u}(p) \right\}.
\end{gather*}%
Under these choices for the Cartan decomposition, the maximal abelian subalgebra $\mathfrak{a}$ of $\mathfrak{s}$ is the subspace of diagonal matrices in $\mathfrak{s}$, so one has
\begin{gather*}
\mathfrak{a}=\left\{\left(
\begin{matrix}
\Lambda & 0 & 0 \\
0 & -\Lambda & 0 \\
0 & 0 & 0%
\end{matrix}
\right) \colon \Lambda =\operatorname{diag}\{a_{1},\ldots,a_{p}\},\ a_{j}\in \mathbb{R}\right\}.
\end{gather*}
Here one sees that the real rank of $\mathfrak{su}(p,q) $ is $p$ if $p\leq q$. A Cartan subalgebra $\mathfrak{h}$ containing $\mathfrak{a}$ is given by diagonal matrices of the form
\begin{gather*}
\left(
\begin{matrix}
D & 0 & 0 \\
0 & -\overline{D} & 0 \\
0 & 0 & {\rm i}T
\end{matrix}
\right),
\end{gather*}
where $D$ has complex entries while $T$ has real entries so that ${\rm i}T$ has pure imaginary entries. The zero trace condition reads as $\operatorname{tr} ( 2\func{Im}D+T ) =0$, so the real part of $D$ is arbitrary. The Cartan subalgebra $\mathfrak{h}$ decomposes as $\mathfrak{h}=\mathfrak{h}_{k}\oplus \mathfrak{a}$ where $\mathfrak{h}_{k}$ is the set of diagonal matrices in $\mathfrak{k}$, that is,
\begin{gather*}
\left(
\begin{matrix}
{\rm i}\Lambda & 0 & 0 \\
0 & {\rm i}\Lambda & 0 \\
0 & 0 & {\rm i}T
\end{matrix}
\right)
\end{gather*}
with $\Lambda $ and $T$ are diagonal matrices such that $\operatorname{tr}(2\Lambda +T)=0$. Therefore, $\dim \mathfrak{h}_{k}=p+(q-p) -1=q-1$ is the rank of $\mathfrak{su}(p,q) $ and $\dim \mathfrak{a}+\dim
\mathfrak{h}_{k}=p+q-1$, which is the rank of $\mathfrak{sl} ( p+q,\mathbb{C} ) $. The Cartan subalgebra $\mathfrak{h}_{\mathbb{C}}$ is the Lie algebra of complex diagonal matrices in $\mathfrak{sl} ( n,\mathbb{C} ) $, $n=p+q$.

In order to give a basis $\{Z_1, \ldots,Z_{p-1}\}$ as in Theorem \ref{teo.baseZ}, we need first a basis $\mathcal{B}$ of $\mathfrak{h }_k$ determined by the roots in $\Sigma_{\mathbb{C},\mathrm{arr}}^\bot$, $\Sigma_{\mathbb{C},\mathrm{arr}}$ and $\Sigma_{\mathrm{Im}}$. So we proceed with the description of the root system.

The roots of $\mathfrak{h}_{\mathbb{C}}$ are the linear functionals given by the differences of the diagonal coordinate functionals in $\mathfrak{h}_{\mathbb{C}}$. To simplify notations, for $H=\operatorname{diag}\{a_{1},\ldots,a_{n}\}\in \mathfrak{h }_\mathbb{C}$ write $\mu _{j}(H) =a_{j}$ if $1\leq j\leq 2p$ and $\theta _{j}(H) =a_{2p+j}$ if $1\leq j\leq q-p$.

As usual in $\mathfrak{sl}( n,\mathbb{C}) $, $\mathfrak{h}_{\mathbb{R}}$ is the subspace of real diagonal matrices. The imaginary roots (annihilating on $\mathfrak{a}$) are $\theta _{j}-\theta _{k}$. In particular, there are no imaginary roots if $p=q$.

The Satake diagram is given by a simple system $\Sigma _{\mathbb{C}}$ of $\mathfrak{h}_{\mathbb{C}}$ such that the imaginary roots in~$\Sigma _{\mathbb{C}}$ span the set of all imaginary roots. Simple systems $\Sigma _{\mathbb{C}}$ are written differently in the cases $p<q$ and $p=q$. If $p<q$ then
\begin{gather*}
\Sigma _{\mathbb{C}} = \{\mu _{1}-\mu _{2},\ldots,\mu _{p-1}-\mu _{p}\}\cup \{\mu _{p}-\theta _{1}\} \\
\hphantom{\Sigma _{\mathbb{C}} =}{} \cup \{\theta _{1}-\theta _{2},\ldots,\theta _{q-p-1}-\theta _{q-p}\}\cup \{\theta _{q-p}-\mu _{2p}\} \cup \{\mu _{2p}-\mu _{2p-1},\ldots,\mu _{p+2}-\mu _{p+1}\},
\end{gather*}
while for $p=q$ we have
\begin{gather*}
\Sigma _{\mathbb{C}}=\{\mu _{1}-\mu _{2},\ldots,\mu _{p-1}-\mu _{p}\}\cup \{\mu _{p}-\mu _{2p}\}\cup \{\mu _{2p}-\mu _{2p-1},\ldots,\mu _{p+2}-\mu_{p+1}\}.
\end{gather*}%
It can be checked that these sets are in fact simple systems of roots with Dynkin diagram $A_{l}$.

To write the restricted system determined by $\mathfrak{a}$ we consider the parametrization by real matrices $\Lambda =\operatorname{diag}\{a_{1},\ldots,a_{p}\}$ in a way that $\mathfrak{a }$ is constituted by the following
matrices
\begin{gather*}
\left(
\begin{matrix}
\Lambda & 0 & 0 \\
0 & -\Lambda & 0 \\
0 & 0 & 0%
\end{matrix}
\right).
\end{gather*}
Denote $\lambda_j(\Lambda)=a_j$.
\begin{enumerate}\itemsep=0pt
\item The simple roots $\theta _{j}-\theta _{j+1}$ are imaginary.
\item For $j=1,\ldots,p-1$, the restrictions to $\mathfrak{a}$ given by the simple roots $\mu _{j}-\mu _{j+1}$ and $\mu _{p+j+1}-\mu _{p+j}$ are equal to $\lambda _{j}-\lambda _{j+1}$.
\item If $p<q$ the roots $\mu _{p}-\theta _{1}$ are $\theta _{q-p}-\mu _{2p}$ restrict to $\lambda _{p}$.
\item If $p=q$ the root $\mu _{p}-\mu _{2p}$ restricts to $2\lambda _{p}$.
\end{enumerate}

The simple system $\Sigma $ obtained by restriction of $\Sigma _{\mathbb{C}}$ is
\begin{gather*}
\Sigma =\{\lambda _{1}-\lambda _{2},\ldots,\lambda _{p-1}-\lambda _{p},\lambda _{p}\}\qquad \text{for} \quad p<q, \quad \text{and}\\
\Sigma =\{\lambda _{1}-\lambda _{2},\ldots,\lambda _{p-1}-\lambda_{p},2\lambda _{p}\}\qquad \text{for} \quad p=q.
\end{gather*}
The Satake diagram of $\Sigma _{\mathbb{C}}$ and the Dynkin diagram of $\Sigma$ in the case $p<q$ are

\centerline{\includegraphics{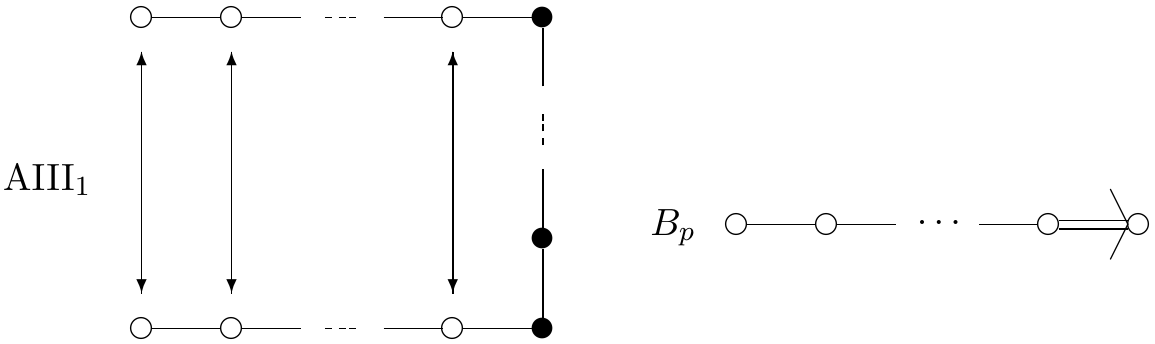}}

\noindent The corresponding diagrams of $\Sigma _{\mathbb{C}}$ and $\Sigma$ in the case $p=q$ are

\centerline{\includegraphics{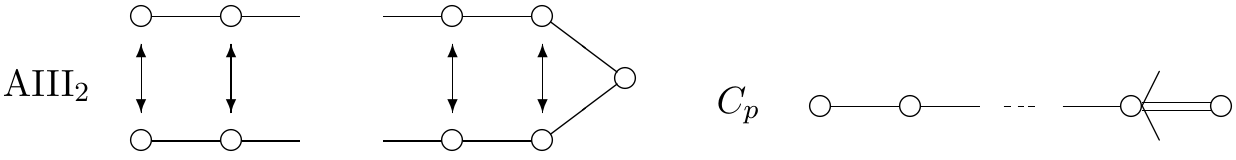}}

The first elements in the basis $\mathcal{B}$ correspond to a basis of $\mathfrak{h }_{\mathrm{arr}}$, and more precisely of $\mathfrak{h }_{\mathrm{arr}}^\bot$ if $p<q$. We shall describe $\Sigma_{\mathbb{C},\mathrm{arr}}$
and $\Sigma_{\mathbb{C},\mathrm{arr}}^\bot$.

The roots in the Satake diagram of $\Sigma _{\mathbb{C}}$ linked by double arrows are given by the following pairs: if $p<q$
\begin{gather*}
\Sigma _{\mathbb{C},\mathrm{arr}} = \{\mu _{1}-\mu _{2},\mu _{p+2}-\mu _{p+1}\}\cup \{\mu _{2}-\mu _{3},\mu _{p+3}-\mu _{p+2}\}\cup \cdots \\
\hphantom{\Sigma _{\mathbb{C},\mathrm{arr}} =}{} \cdots \cup \{\mu _{p-1}-\mu _{p},\mu _{2p}-\mu _{2p-1}\}\cup \{\mu
_{p}-\theta _{1},\theta _{q-p}-\mu _{2p}\},
\end{gather*}
while for $p=q$,
\begin{gather*}
\Sigma _{\mathbb{C},\mathrm{arr}} =\{\mu _{1}-\mu _{2},\mu _{p+2}-\mu
_{p+1}\}\cup \{\mu _{2}-\mu _{3},\mu _{p+3}-\mu _{p+2}\}\cup \!\cdots\!
\cup \{\mu _{p-1}\!-\mu _{p},\mu _{2p}-\mu _{2p-1}\}.
\end{gather*}
When $p<q$, the set of imaginary roots is $\Sigma _{\func{Im}}=\{\theta_{1}-\theta _{2},\ldots,\theta _{q-p-1}-\theta _{q-p}\}$, so the subset of pairs of roots which are orthogonal to $\Sigma _{\func{Im}}$ is $\Sigma _{\mathbb{C},\mathrm{arr}}^{\bot }= \Sigma _{\mathbb{C}\mathrm{,arr}}\setminus\{\mu _{p}-\theta _{1},\theta _{q-p}-\mu _{2p}\}$. If $p=q$ then $\Sigma _{\mathbb{C}\mathrm{,arr}}=\Sigma _{\mathbb{C}\mathrm{,arr}}^{\bot }$ since there are no imaginary roots. We give a basis of $\mathfrak{h}_{\mathrm{arr}}^{\bot }$ to fill $\mathcal{B}$ which is valid in both cases.

Up to normalization, $H_{\mu _{j}-\mu _{j+1}}$ is given by the diagonal matrix $\Lambda $ with non-zero entries $1$ and $-1$ in positions $j$ and $j+1$ respectively, $j=1, \ldots,2p-1$. Given the pair of roots $%
\gamma_j=\mu_j-\mu_{j+1}$, $\delta_j=\mu_{p+j+1}-\mu_{p+j}$ in $\Sigma _{\mathbb{C},\mathrm{arr}}^\bot$, the elements spanning $\mathfrak{h }_{\mathrm{arr}}^\bot$ are ${\rm i}H_{\delta-\gamma}={\rm i} H_{\mu _{j}-\mu _{j+1}-(\mu_{p+j+1}-\mu _{p+j})}={\rm i}H_{\mu _{j}-\mu _{j+1}}+{\rm i}H_{\mu _{p+j}-\mu _{p+j+1}}$. That is, for $p\leq q$,
\begin{gather}
\mathfrak{h}_{\mathrm{arr}}^{\bot }=\left\{\left(
\begin{matrix}
{\rm i}D & 0 & 0 \\
0 & {\rm i}D & 0 \\
0 & 0 & 0
\end{matrix}
\right)\colon \operatorname{tr}D=0\right\}. \label{formatrizhortogonal}
\end{gather}
The first $p-1$ elements in the basis $\mathcal{B}$ are multiples
of
\begin{gather*}
{\rm i}H_{\mu _{j}-\mu _{j+1}+\mu _{p+j}-\mu _{p+j+1}}=\left(
\begin{matrix}
{\rm i}D_{j,j+1} & 0 & 0 \\
0 & {\rm i}D_{j,j+1} & 0 \\
0 & 0 & 0
\end{matrix}
\right)
\end{gather*}
with $D_{j,j+1}=\operatorname{diag}\{0,\ldots,1_{j},-1_{j+1},\ldots,0\}$. In the case $p=q$ these elements are enough to complete $\mathcal{B}$ since there are no imaginary roots.

Suppose that $p<q$, that is, we are in the $\mathrm{AIII}_{1}$ case. The last elements of $\mathcal{B}$ constitute a~basis of $\mathfrak{h }_{\mathrm{Im}}$. We have $\Sigma_{\func{Im}} =\{\theta_1-\theta_2,\ldots,\theta_{q-p-1}-\theta_{q-p}\}$, therefore
\begin{gather*}
\mathfrak{k}_{\func{Im}} = \left\{\left(
\begin{matrix}
0 & 0 & 0 \\
0 & 0 & 0 \\
0 & 0 & Z
\end{matrix}
\right)\colon Z\in \mathfrak{su}(q-p) \right\}, \\
\mathfrak{h}_{\func{Im}} = \left\{\left(
\begin{matrix}
0 & 0 & 0 \\
0 & 0 & 0 \\
0 & 0 & {\rm i}T
\end{matrix}
\right)\colon \operatorname{tr}T=0 \right\},
\end{gather*}
with $T$ diagonal with real entries. Up to normalization ${\rm i}H_{\theta_j-\theta_{j+1}}\in \mathfrak{h }_{\mathrm{Im}}$ is given by a~matrix as above with $T$ diagonal with $1$ in position $j$ and $-1$ in position $j+1$, for $j=1, \ldots,q-p-1$.

The center $\mathfrak{z }(\mathfrak{m })$ is $\mathfrak{h}_{\func{Im}}^{\bot}$, the orthogonal of $\mathfrak{h}_{\func{Im}}$ in $\mathfrak{h}_{k}$ (see Proposition~\ref{propCentrom}), so we have
\begin{gather*}
\mathfrak{h}_{\func{Im}}^{\bot }=\left\{\left(
\begin{matrix}
{\rm i}D & 0 & 0 \\
0 & {\rm i}D & 0 \\
0 & 0 & {\rm i}a\, \mathrm{Id}
\end{matrix}
\right)\colon 2\operatorname{tr}D+(q-p) a=0 \right\}.
\end{gather*}

One last element in $\mathcal{B}$ is missing since the inclusion $\mathfrak{h}_{\mathrm{arr}}^{\bot }\subset \mathfrak{h}_{\func{Im}}^{\bot }$ is strict and of codimension one. The remaining element is a non zero multiple of ${\rm i}H_{\mu_p-\theta_1-(\theta_{q-p}-\mu_{2p})}$ since $\{\mu_p-\theta_1,\theta_{q-p}-\mu_{2p}\}$ is the only pair of complex roots in $\Sigma_{\mathbb{C},\mathrm{arr}}$ linked to imaginary roots. The matrices corresponding to $H_{\mu_p-\theta_1}$ and $H_{\theta_{q-p}-\mu_{2p}}$ are diagonal matrices with non zero entries being $1$ and $-1$ in positions $p$, $2p+1$ and $p+q$, $2p$, respectively. Then ${\rm i}H_{\mu_p-\theta_1-(\theta_{q-p}-\mu_{2p})}={\rm i}(H_{\mu_p-\theta_1}-H_{\theta_{q-p}-\mu_{2p}})$.

The computations above account to
\begin{gather*}
\mathcal{B}= \big\{ {\rm i}H_{\mu_{j}-\mu _{j+1}+\mu _{p+j}-\mu _{p+j+1}}^N\big\}_{j=1}^{p-1} \cup\big\{{\rm i}H_{\mu_p-\theta_1+\mu_{2p}-\theta_{q-p}}^N\big\} \cup\big\{{\rm i}H_{\theta_j-\theta_{j+1}}^N\big\}_{j=1}^{q-p-1}.
\end{gather*}

Let $\mathcal{B}^{\bot }$ be the dual basis of $\mathcal{B}$ with respect to the Cartan--Killing form in $\mathfrak{sl}(p+q,\mathbb{C})$. The first $p-1$ elements of $\mathcal{B}^{\bot }$ are the elements $Z_{1},\ldots,Z_{p-1}$ appearing in Theorem~\ref{teo.baseZ}. These are, up to normalization,
\begin{gather*}
Z_{j}=\left(
\begin{matrix}
{\rm i}E_{j} & 0 & 0 \\
0 & {\rm i}E_{j} & 0 \\
0 & 0 & -{\rm i}a\func{Id}
\end{matrix}
\right), \qquad j=1,\ldots,p-1,
\end{gather*}%
where $E_{j}=\operatorname{diag}\{b,\ldots,b,-a,\ldots,-a\}$, the last $b$ is in position $j$, and $a,b\in \mathbb{R}$ verify $2j-a(p+q) =0$ and $b+a=1$, that is, $a=2j/(p+q)$, $b=(p+q-2j)/(p+q)$. Through the Weil construction, $\{f_{Z_{1}},\ldots,f_{Z_{p-1}}\}$ is a basis of $H^{2}(\mathbb{F},\mathbb{R})$.

For a partial flag manifold $\mathbb{F}_{\Theta }$ with $\Theta \subset \Sigma $ the $2$-homology $H_{2}(\mathbb{F}_{\Theta },\mathbb{R})$ is spanned by the Schubert cells $S_{r_{\alpha }}$ with $\alpha $ running
through the rank $2$ simple roots in $\Sigma $ outside $\Theta $. Hence as in the case of the maximal flag manifold we get a basis of $H^{2}(\mathbb{F}_{\Theta },\mathbb{R})$ of the form $\{f_{Z_{j_{1}}},\ldots,f_{Z_{j_{s}}}\}$ where $j_{1},\ldots,j_{s}$ are the indices corresponding to the rank~$2$ roots in $\Sigma \setminus \Theta $ (long roots if $p<q$ and short roots if $p=q$). This basis is dual to the Schubert cells.

\subsection*{Acknowledgements}

V.~del Barco supported by FAPESP grants 2015/23896-5 and 2017/13725-4. L.A.B.~San Martin supported by CNPq grant 476024/2012-9 and FAPESP grant 2012/18780-0. The authors express their gratitude to Lonardo Rabelo for careful reading a previous version of this manuscript and his useful suggestions.

\pdfbookmark[1]{References}{ref}
\LastPageEnding

\end{document}